\documentclass[11pt,reqno]{amsart}
\usepackage[margin=1.15in]{geometry}
\usepackage{amsmath,amssymb,amsthm}
\usepackage{enumitem}
\usepackage{tikz}
\usetikzlibrary{decorations.markings}
\usepackage[colorlinks=true,linkcolor=blue,citecolor=blue,urlcolor=blue]{hyperref}

\theoremstyle{plain}
\newtheorem{theorem}{Theorem}[section]
\newtheorem{lemma}[theorem]{Lemma}
\newtheorem{corollary}[theorem]{Corollary}
\theoremstyle{definition}
\newtheorem{definition}[theorem]{Definition}
\newtheorem{example}[theorem]{Example}

\numberwithin{equation}{section}

\newcommand{\ZZ}{\mathbb{Z}}
\newcommand{\kk}{K}
\newcommand{\reg}{\operatorname{reg}}
\newcommand{\pd}{\operatorname{pd}}
\newcommand{\dpt}{\operatorname{depth}}
\newcommand{\Ind}{\operatorname{Ind}}
\newcommand{\Cl}{\operatorname{Cl}}
\newcommand{\lk}{\operatorname{lk}}
\newcommand{\indm}{\nu}
\newcommand{\Hred}{\widetilde{H}}
\newcommand{\Gnm}{G_{n,m}}
\newcommand{\Dnm}{\Delta_{n,m}}
\newcommand{\ball}[1]{B(#1)}
\newcommand{\cB}{\mathcal{B}}
\newcommand{\cM}{\mathcal{M}}
\newcommand{\cP}{\mathcal{P}}
\newcommand{\cQ}{\mathcal{Q}}

\begin{document}

\title[Edge ideals of complements of a class of circular interval graphs]
{Edge ideals of complements of a class of circular interval graphs}

\author[B. A. Wani]{Bilal Ahmad Wani$^*$}
\address{
Department of Mathematics\\
National Institute of Technology\\
Srinagar-190006\\
India}
\email{bilalwanikmr@gmail.com}
\thanks{$^*$ Corresponding author}

\begin{abstract}
Let $C_n^m$ be the $m$-th power of the cycle on $n$ vertices, and let
$\Gnm$ be its complement, where $m\ge1$ and $n\ge3m+1$. We prove that every induced
subcomplex of the clique complex of $C_n^m$ is homotopy equivalent either to
a disjoint union of contractible complexes or to a circle, and we establish
this dichotomy for a larger class of circular interval graphs. As a
consequence, we obtain a closed formula for the linear strand of the Betti
table of $R/I(\Gnm)$. We also prove that $R/I(\Gnm)$ is a Buchsbaum ring of
depth $2$ and that the independence complex of $\Gnm$ is a triangulated
$m$-manifold with boundary for $m\ge2$, namely a M\"obius band or an
annulus when $m=2$. For the
complement $\overline H$ of every graph $H$ in the above class, we show
that $\reg(I(\overline H)^k)=2k+\indm(\overline H)-1$ for all $k\ge2$, where
$\indm$ denotes the induced matching number. In
particular, some power of $I(\Gnm)$ has a linear resolution if and only if
$n\ge4m+1$. In this range, we further prove that the colon ideals
$(I^{k+1}:M)$, where $I=I(\Gnm)$ and $M$ is a minimal monomial generator of
$I^k$, have regularity $2$.
\end{abstract}

\subjclass[2020]{13F55, 13D02, 13H10, 05E40, 05E45, 05C69}
\keywords{Edge ideal, Castelnuovo--Mumford regularity, powers of ideals,
independence complex, circular interval graph, Buchsbaum ring, induced
matching}

\maketitle

\section{Introduction}

Let $\kk$ be a field and let $R=\kk[x_0,\dots,x_{n-1}]$ be the polynomial ring
in $n$ variables over $\kk$. To every finite simple graph $G$ on the vertex
set $\{0,\dots,n-1\}$ one associates its edge ideal
\[
I(G)=(x_ux_v:\{u,v\}\in E(G))\subseteq R.
\]
The regularity of edge ideals and of their powers has been studied by many
authors. Fr\"oberg \cite{Fro} proved that $I(G)$ has a linear resolution if
and only if the complement $\overline G$ is chordal. Beyarslan, H\`a and
Trung \cite{BHT} showed that $\reg(I(G)^k)\ge2k+\indm(G)-1$ for every
$k\ge1$, where $\indm(G)$ denotes the induced matching number of $G$, while
the upper bound $\reg(I(G)^k)\le2k+\reg I(G)-2$ remains conjectural. Nevo
and Peeva \cite{NP} asked whether $\reg(I(G)^k)=2k$ for all $k\ge2$ whenever
$G$ is gap-free with $\reg I(G)=3$. Banerjee, Beyarslan and H\`a \cite{BBH}
called $G$ locally linear if $\reg(I(G):x)\le2$ for every vertex $x$, and
they proved that locally linear graphs satisfy the conjectured upper bound
and that gap-free locally linear graphs satisfy $\reg(I(G)^k)=2k$ for every
$k\ge2$. Minh and Vu \cite{MV} settled the cases $k=2$ and $k=3$ for every
gap-free graph of regularity three.

In this paper, we study these questions for the family
\[
\Gnm=\overline{C_n^m},\qquad I=I(\Gnm),
\]
where $C_n^m$ is the $m$-th power of the cycle on $\ZZ_n$, $m\ge1$ and
$n\ge3m+1$.
This is exactly the range in which the maximal cliques of $C_n^m$ are the
sets of $m+1$ cyclically consecutive vertices. The graded Betti numbers of
$R/I$ in this range were studied by Rather, Pirzada and Aijaz \cite{RPA}.
They computed the second strand of the Betti table, showed that
$\reg(R/I)=2$ and $\pd(R/I)=n-2$, obtained the linear strand recursively,
and proved that $\indm(\Gnm)$ equals $2$ for $3m+1\le n\le4m$ and $1$ for
$n\ge4m+1$. For $m=2$ these results are due to Rather, Pirzada and Singh
\cite{RPS}.

Our first main result, Theorem \ref{thm:dichotomy}, concerns graphs $H$ on
$\ZZ_N$ whose maximal cliques are arcs and which satisfy
$|N_H[v]|+\omega(H)\le N+1$ at every vertex $v$. These are circular interval
graphs in the sense of Chudnovsky and Seymour \cite{CS}, and for $H=C_n^m$
the condition is equivalent to $n\ge3m+1$. Clique complexes of powers of
cycles and of families of arcs have been studied in \cite{Ad,AA,AAFPP},
and the homotopy type of $\Cl(C_n^m)$ itself is known in this range.
Hochster's formula, however, requires control of every induced subcomplex.
We prove that, for every nonempty $W\subseteq\ZZ_N$, the complex
$\Cl(H)[W]$ is homotopy equivalent to a disjoint union of $c(W)$
contractible complexes when $c(W)\ge1$ and to $S^1$ when $c(W)=0$, where
$c(W)$ is the number of cyclically consecutive pairs of $W$ whose
connecting arcs are not cliques of $H$. For $H=C_n^m$, the case $c(W)=0$
and the vanishing of higher homology were obtained in \cite{RPA}. Combining
the dichotomy with Hochster's formula, we obtain in Theorem
\ref{thm:linstrand} the closed formula
\[
\beta_{k-1,k}(R/I)=n\binom{n-m-1}{k-1}-\binom nk+\frac nkC_m(n,k),
\]
where $C_m(n,k)$ is the number of compositions of $n$ into $k$ parts
belonging to $\{1,\dots,m\}$. In Theorem \ref{thm:buchsbaum}, we prove that
$R/I$ is a Buchsbaum ring of depth $2$ which is Cohen--Macaulay if and only
if $m=1$, and that $\Ind(\Gnm)$ is a triangulated $m$-manifold, with
boundary for $m\ge2$. For $m=2$ it is a M\"obius band if $n$ is odd and an
annulus if $n$ is even.

The second part of the paper concerns powers of edge ideals. In Theorem
\ref{thm:gen-loclin}, we prove that the complement $\overline H$ of every
graph $H$ in the above class is locally linear. Together with the results
of \cite{BBH} and \cite{BHT}, this gives
$\reg(I(\overline H)^k)=2k+\indm(\overline H)-1$ for every $k\ge2$ in
Theorem \ref{thm:gen-powers}. For $\Gnm$, we obtain in Theorem
\ref{thm:main} that
\[
\reg(I^k)=
\begin{cases}
2k+1,&3m+1\le n\le4m,\\
2k,&n\ge4m+1,
\end{cases}
\]
for every $k\ge2$, and in the first range for every $k\ge1$. In
particular, some power of $I$ has a linear resolution if and only if
$n\ge4m+1$. In this range, we study the even-connection graphs of Banerjee
\cite{Ba} and prove in Theorem \ref{thm:key} that $\reg(I^{k+1}:M)=2$ for
every minimal monomial generator $M$ of $I^k$, with an explicit perfect
elimination ordering obtained from the circular order on $\ZZ_n$. Together
with an inequality of Banerjee \cite[Theorem 5.2]{Ba}, this gives a direct
proof of the equality $\reg(I^k)=2k$ for $k\ge2$.

Several special cases are known. For $m=1$, the graphs $G_{n,1}$ are the
complements of cycles. Biermann \cite{Bi} constructed a cellular minimal
free resolution of $I(G_{n,1})$, and Basser, Diethorn, Miranda and
Stinson-Maas \cite{BDMS} proved that $I(G_{n,1})^k$ has linear quotients
for every $k\ge2$. The graph $\Gnm$ is the circulant graph
$C_n(m+1,\dots,\lfloor n/2\rfloor)$, and graded Betti numbers of several
families of circulant graphs were computed by Anand and Roy \cite{AR}.
Finally, $\Gnm$ is cubic exactly when $(n,m)$ is $(6,1)$, $(8,2)$ or
$(10,3)$, and in these cases our values agree with those of Hang, Pham and
Vu \cite{HPV} for cubic circulant graphs.

This paper is organized as follows. In Section \ref{sec:prelim}, we recall
the necessary background. In Section \ref{sec:dichotomy}, we prove the
homotopy dichotomy, and in Section \ref{sec:betti} we derive its
consequences for the Betti numbers and the Buchsbaum and manifold
structure. Section \ref{sec:powers} treats the regularity of powers, and
Section \ref{sec:colon} the colon ideals for $n\ge4m+1$.

\section{Preliminaries}\label{sec:prelim}

In this section, we collect the definitions and known results that will be
used in the rest of the paper. Throughout, all graphs are finite and simple,
$\kk$ is a field, and $R$ denotes the polynomial ring over $\kk$ in the
variables $x_v$, where $v$ runs over the vertex set of the graph under
consideration. For the graphs $\Gnm$, this is the ring
$R=\kk[x_0,\dots,x_{n-1}]$ of the introduction.

Let $G$ be a graph with vertex set $V$. The edge ideal of $G$ is
$I(G)=(x_ux_v:\{u,v\}\in E(G))$, and the complement of $G$ is denoted by
$\overline{G}$. We write $\Ind(G)$ for the independence complex of $G$, that
is, the simplicial complex on $V$ whose faces are the independent sets of
$G$. Recall that $I(G)$ is the Stanley--Reisner ideal of $\Ind(G)$ and that
$\Ind(G)=\Cl(\overline{G})$, where $\Cl(\overline{G})$ denotes the clique
complex of $\overline{G}$. For a simplicial complex $\Delta$ and a subset
$W\subseteq V$, we denote by $\Delta[W]$ the subcomplex of $\Delta$ induced
on $W$. For a face $F$ of $\Delta$, the link of $F$ is
$\lk_\Delta(F)=\{S:S\cap F=\emptyset,\ S\cup F\in\Delta\}$. An
\emph{induced matching} of $G$ is a set of pairwise disjoint edges of $G$
such that no edge of $G$ joins two of them. The induced
matching number $\indm(G)$ is the largest number of edges in an induced
matching of $G$, and $G$ is called \emph{gap-free} if $\indm(G)=1$. A graph
is \emph{chordal} if every cycle of length at least four has a chord. A
vertex of a graph is \emph{simplicial} if its neighbors form a clique. An
ordering $v_1,\dots,v_N$ of the vertices is a \emph{perfect elimination
ordering} if each $v_i$ is simplicial in the subgraph induced by
$\{v_i,\dots,v_N\}$, and a graph is chordal if and only if it admits a
perfect elimination ordering (see \cite[Chapter 9]{HHbook}).
Following \cite{BBH}, we call $G$ \emph{locally linear} if
$\reg(I(G):x)\le2$ for every vertex $x$. Regularity always means
Castelnuovo--Mumford regularity, and for a nonzero graded ideal $I$ of $R$
generated in positive degrees we use $\reg(R/I)=\reg I-1$.

We now recall some known results. The first
is Hochster's formula, which expresses the graded Betti numbers of a
Stanley--Reisner ring in terms of the reduced homology of induced
subcomplexes.

\begin{theorem}[{\cite[Corollary 5.12]{MS}}]\label{thm:hochster}
Let $\Delta$ be a simplicial complex on $V$. Then, for every pair $i,j$,
\[
\beta_{i,j}(R/I_\Delta)
=\sum_{W\subseteq V,\ |W|=j}\dim_\kk\Hred_{j-i-1}(\Delta[W];\kk).
\]
\end{theorem}

The following theorem of Fr\"oberg characterizes the edge ideals with linear
resolution.

\begin{theorem}[{\cite{Fro}}]\label{thm:froberg}
Let $G$ be a graph with at least one edge. Then $\reg I(G)=2$ if and only if
$\overline{G}$ is a chordal graph.
\end{theorem}

Katzman proved the following lower bound for the regularity of an edge
ideal.

\begin{theorem}[{\cite{Ka}}]\label{thm:katzman}
For every graph $G$, we have $\reg(R/I(G))\ge\indm(G)$.
\end{theorem}

Beyarslan, H\`a and Trung extended this inequality to all powers.

\begin{theorem}[{\cite[Theorem 4.5]{BHT}}]\label{thm:bht}
Let $G$ be a graph. Then $\reg(I(G)^k)\ge2k+\indm(G)-1$ for every $k\ge1$.
\end{theorem}

The next theorem summarizes the results of Banerjee, Beyarslan and H\`a on
locally linear graphs.

\begin{theorem}[{\cite[Theorems 3.6 and 4.5]{BBH}}]\label{thm:bbh}
Let $G$ be a locally linear graph and let $I=I(G)$. Then
$\reg(I^k)\le2k+\reg I-2$ for every $k\ge1$. If in addition $G$ is gap-free,
then $\reg(I^k)=2k$ for every $k\ge2$.
\end{theorem}

A simplicial complex is \emph{pure} if all its facets have the same
dimension, and a face of a pure complex $\Delta$ has \emph{codimension one}
if its dimension is $\dim\Delta-1$. A \emph{shelling} of a pure complex is
an ordering $F_0,\dots,F_s$ of its facets such that, for every
$1\le p\le s$, the intersection $F_p\cap\bigcup_{q<p}F_q$ is a nonempty
union of codimension-one faces of $F_p$. A pure complex is \emph{shellable}
if its facets admit a shelling.

We shall use the following criterion of Schenzel for the Buchsbaum property
of a Stanley--Reisner ring.

\begin{theorem}[{\cite[Chapter II]{StV}}]\label{thm:schenzel}
Let $\Delta$ be a simplicial complex. Then $\kk[\Delta]$ is Buchsbaum if and
only if $\Delta$ is pure and $\Hred_i(\lk_\Delta F;\kk)=0$ for every
nonempty face $F$ of $\Delta$ and every $i<\dim\lk_\Delta F$. In this case,
\[
\dpt\kk[\Delta]
=\min\{\dim\kk[\Delta],\ \min\{i+1:\Hred_i(\Delta;\kk)\ne0\}\}.
\]
\end{theorem}

The following theorem of Danaraj and Klee identifies balls among
shellable complexes.

\begin{theorem}[{\cite{DK}}]\label{thm:dk}
Let $\Delta$ be a pure shellable simplicial complex in which no face of
codimension one belongs to more than two facets. Then $\Delta$ is either a
ball or a sphere, and it is a ball as soon as some face of codimension one
belongs to only one facet.
\end{theorem}

We next recall the notion of even-connectedness, introduced by Banerjee,
which describes the generators of the colon ideals $(I^{k+1}:M)$.

\begin{definition}[{\cite[Definition 6.2]{Ba}}]\label{def:evenconn}
Let $M=e_1\cdots e_k$ be a product of edges of a graph $G$. Two vertices $u$
and $v$, not necessarily distinct, are \emph{even-connected with respect to
$M$} if there exist $\ell\ge1$ and vertices $p_0=u,p_1,\dots,p_{2\ell+1}=v$
such that the following conditions hold.
\begin{enumerate}[label=\textup{(\roman*)}]
\item $\{p_i,p_{i+1}\}\in E(G)$ for every $0\le i\le2\ell$.
\item For each $1\le i\le\ell$, the edge
$\{p_{2i-1},p_{2i}\}$ is one of the $e_j$.
\item No edge $e_j$ is used in (ii) more often than its multiplicity in $M$.
\end{enumerate}
\end{definition}

The following theorem of Banerjee makes this description precise.

\begin{theorem}[{\cite[Theorems 6.1 and 6.7]{Ba}}]\label{thm:banerjeecolon}
Let $I=I(G)$, let $k\ge1$, and let $M$ be a minimal monomial generator of
$I^k$. Then $(I^{k+1}:M)$ is generated in degree two. Moreover,
$x_ux_v$ is a minimal generator of $(I^{k+1}:M)$ if and only if either
$\{u,v\}\in E(G)$ or $u$ and $v$ are even-connected with respect to $M$.
\end{theorem}

We shall also use the fact that polarization preserves regularity
\cite[Corollary 1.6.3]{HHbook}. In the situation of Theorem
\ref{thm:banerjeecolon}, set $J=(I^{k+1}:M)$. Let $G'$ be the graph on $V$
whose edges are the pairs $\{u,v\}$ with $u\ne v$ and $x_ux_v\in J$, and let
$S=\{u:x_u^2\in J\}$. Since $J$ is generated in degree two, it is
generated by the monomials $x_ux_v$ with $\{u,v\}\in E(G')$ and $x_u^2$ with
$u\in S$. Polarization replaces each $x_u^2$ by $x_ux_{u'}$ for a new
variable $x_{u'}$. Hence the polarization $J^{\mathrm{pol}}$ is the edge
ideal of the graph $G''$ obtained from $G'$ by attaching a pendant vertex
$u'$ to each $u\in S$.

\section{The homotopy dichotomy}\label{sec:dichotomy}

In this section, we prove Theorem
\ref{thm:dichotomy}, which determines the homotopy type of every induced
subcomplex of the clique complex of a graph in the following class.

\begin{definition}\label{def:class}
Let $N\ge2$, and let $H$ be a graph on $\ZZ_N$ whose maximal cliques are
\emph{arcs}, that is, sets of cyclically consecutive elements of $\ZZ_N$.
Write $\omega=\omega(H)$ for its clique number. We assume that
\begin{equation}\label{eq:star}
|N_H[v]|+\omega\le N+1\qquad\text{for every }v\in\ZZ_N. \tag{$\ast$}
\end{equation}
\end{definition}

Every graph $H$ as in Definition \ref{def:class} is a circular interval
graph in the sense of Chudnovsky and Seymour \cite{CS}, with its maximal
cliques giving the circular interval representation. We do not use the
converse, since a circular interval representation need not have the
prescribed arcs as its maximal cliques. The hypotheses of Definition
\ref{def:class} refer to the chosen labeling of the vertices by $\ZZ_N$. In
other words, we require that there exists some cyclic order in which the
maximal cliques are arcs and \eqref{eq:star} holds, and every statement
below is read with such an order fixed. The conclusions that we draw from
these hypotheses, namely the homotopy types of induced subcomplexes and the
regularities of powers, depend only on the isomorphism class of $H$.

The assumption $N\ge2$ excludes only the graph $H=K_1$, which satisfies
\eqref{eq:star} but has $\omega=N$. The graphs
$C_n^m$, whose maximal cliques for $n\ge3m+1$ are the arcs of $m+1$
consecutive vertices by Lemma \ref{lem:windows} below, satisfy $\omega=m+1$
and $|N_H[v]|=2m+1$. Hence \eqref{eq:star} becomes $3m+2\le n+1$, that is,
$n\ge3m+1$, and the standing assumption for $C_n^m$ is exactly
\eqref{eq:star}. Figure \ref{fig:g92} shows the graphs $C_9^2$ and
$G_{9,2}$.

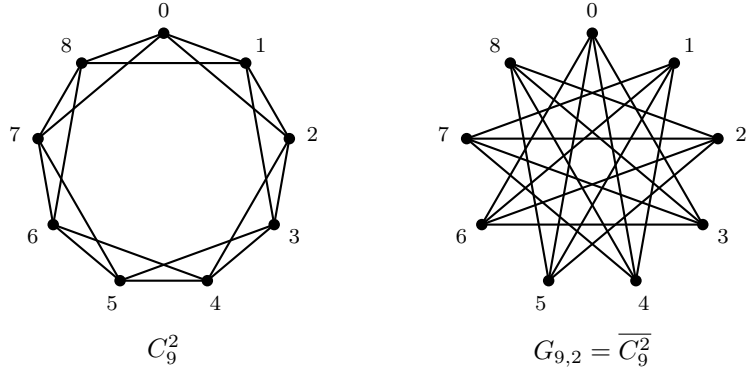
\begin{figure}[ht]
\centering
\begin{tikzpicture}[scale=1.35]
\begin{scope}
\foreach \i in {0,...,8}{
  \coordinate (a\i) at ({90-40*\i}:1.25);
}
\foreach \i in {0,...,8}{
  \pgfmathtruncatemacro{\j}{mod(\i+1,9)}
  \pgfmathtruncatemacro{\k}{mod(\i+2,9)}
  \draw[thick] (a\i) -- (a\j);
  \draw[thick] (a\i) -- (a\k);
}
\foreach \i in {0,...,8}{
  \fill (a\i) circle (1.6pt);
  \node[font=\scriptsize] at ({90-40*\i}:1.48) {$\i$};
}
\node[font=\small] at (0,-1.85) {$C_9^2$};
\end{scope}
\begin{scope}[xshift=4.2cm]
\foreach \i in {0,...,8}{
  \coordinate (b\i) at ({90-40*\i}:1.25);
}
\foreach \i in {0,...,8}{
  \pgfmathtruncatemacro{\j}{mod(\i+3,9)}
  \pgfmathtruncatemacro{\k}{mod(\i+4,9)}
  \draw[thick] (b\i) -- (b\j);
  \draw[thick] (b\i) -- (b\k);
}
\foreach \i in {0,...,8}{
  \fill (b\i) circle (1.6pt);
  \node[font=\scriptsize] at ({90-40*\i}:1.48) {$\i$};
}
\node[font=\small] at (0,-1.85) {$G_{9,2}=\overline{C_9^2}$};
\end{scope}
\end{tikzpicture}
\caption{The square of the $9$-cycle and its complement. Vertices at cyclic
distance at most $2$ are adjacent on the left, and vertices at cyclic distance
at least $3$ are adjacent on the right.}
\label{fig:g92}
\end{figure}

Let $W=\{w_0,\dots,w_{s-1}\}\subseteq\ZZ_N$ be nonempty, with its vertices in
cyclic order. For $t\in\ZZ_s$ denote by $[w_t,w_{t+1}]$ the arc traversed from
$w_t$ to $w_{t+1}$ in the positive direction. When $s=1$, this arc is all of
$\ZZ_N$. The \emph{gap} $g_t$ of the $t$-th pair is the number of steps from
$w_t$ to $w_{t+1}$ in the positive direction. Thus $g_t\in\{1,\dots,N\}$ and
$\sum_{t\in\ZZ_s}g_t=N$. We call the $t$-th link \emph{covered} if
$[w_t,w_{t+1}]$ is a clique of $H$, and we set
\[
c(W)=\#\{t:\text{the $t$-th link is not covered}\}.
\]
A covered link necessarily joins adjacent vertices, since the forward arc
contains both of its endpoints. When $H=C_n^m$, a link is covered exactly
when its gap is at most $m$, so $c(W)$ counts the cyclic gaps of $W$
exceeding $m$. Figure \ref{fig:cw}
shows an example with two uncovered links.

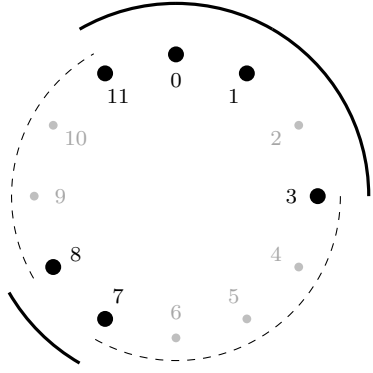
\begin{figure}[ht]
\centering
\begin{tikzpicture}[scale=1.5]
\foreach \i in {0,...,11}{
  \coordinate (v\i) at ({90-30*\i}:1.25);
}
\foreach \i in {2,4,5,6,9,10}{
  \fill[gray!50] (v\i) circle (1.1pt);
  \node[font=\scriptsize,gray!70] at ({90-30*\i}:1.02) {$\i$};
}
\foreach \i in {0,1,3,7,8,11}{
  \fill (v\i) circle (2pt);
  \node[font=\scriptsize] at ({90-30*\i}:1.02) {$\i$};
}
\draw[very thick] ({90-30*11}:1.7) arc ({90-30*11}:{90-30*12}:1.7);
\draw[very thick] (90:1.7) arc (90:{90-30*1}:1.7);
\draw[very thick] ({90-30*1}:1.7) arc ({90-30*1}:{90-30*3}:1.7);
\draw[very thick] ({90-30*7}:1.7) arc ({90-30*7}:{90-30*8}:1.7);
\draw[dashed] ({90-30*3}:1.45) arc ({90-30*3}:{90-30*7}:1.45);
\draw[dashed] ({90-30*8}:1.45) arc ({90-30*8}:{90-30*11}:1.45);
\end{tikzpicture}
\caption{The set $W=\{0,1,3,7,8,11\}$ in $\ZZ_{12}$ with $m=2$. The four
covered links (solid outer arcs) have gaps at most $2$, the two uncovered
links (dashed arcs) have gaps $4$ and $3$, so $c(W)=2$, and $W$ splits into
the runs $\{11,0,1,3\}$ and $\{7,8\}$.}
\label{fig:cw}
\end{figure}

We begin with the following well-known lemma, which describes the effect of
removing a single vertex from a simplicial complex. We shall use it
repeatedly.

\begin{lemma}\label{lem:deletion}
Let $\Delta$ be a simplicial complex and let $v$ be a vertex of $\Delta$.
Write $\Delta\setminus v$ for the induced subcomplex on the remaining
vertices. Then
\[
\Delta=\bigl(v*\lk_\Delta(v)\bigr)\cup(\Delta\setminus v),
\qquad
\bigl(v*\lk_\Delta(v)\bigr)\cap(\Delta\setminus v)=\lk_\Delta(v).
\]
Consequently, the following statements hold.
\begin{enumerate}[label=\textup{(\arabic*)}]
\item If $\lk_\Delta(v)$ is contractible, then
$\Delta\simeq\Delta\setminus v$.
\item If $\lk_\Delta(v)\simeq S^0$ and $\Delta\setminus v$ is contractible,
then $\Delta\simeq S^1$.
\end{enumerate}
\end{lemma}

\begin{proof}
Note that a face of $\Delta$ containing $v$ lies in $v*\lk_\Delta(v)$, and a
face of $\Delta$ not containing $v$ lies in $\Delta\setminus v$. This proves
the displayed decomposition, and the displayed intersection follows directly
from the definitions. Throughout this proof, we write $|\cdot|$ for geometric
realizations. In the rest of the paper, we omit the bars. We use the gluing
lemma (see \cite{Bj}). It states that if a continuous map between the
geometric realizations of two complexes, each written as the union of two
subcomplexes, carries pieces to pieces, and if its restrictions to the two
pieces and to their intersections are homotopy equivalences, then the map
itself is a homotopy equivalence. Note that all decompositions used below
are by subcomplexes. Hence the gluing lemma applies to them. Note also that
the cone $v*\lk_\Delta(v)$ is contractible.

(1) Consider the inclusion map $\varphi\colon\Delta\setminus v\to\Delta$.
Write $\Delta\setminus v=\lk_\Delta(v)\cup(\Delta\setminus v)$ and
$\Delta=\bigl(v*\lk_\Delta(v)\bigr)\cup(\Delta\setminus v)$. Then the
intersection of the two pieces is $\lk_\Delta(v)$ on both sides, and
$\varphi$ carries pieces to pieces. The restriction
$\lk_\Delta(v)\to v*\lk_\Delta(v)$ is a map between contractible complexes.
Hence it is a homotopy equivalence. The other two restrictions are identity
maps. Therefore, by the gluing lemma, $\varphi$ is a homotopy equivalence.

(2) Realize the unreduced suspension of the link as the simplicial
complex
\[
\operatorname{susp}\lk_\Delta(v)
=\bigl(c_+*\lk_\Delta(v)\bigr)\cup\bigl(c_-*\lk_\Delta(v)\bigr),
\]
that is, as two cones glued along $\lk_\Delta(v)$. We define a continuous map
$\psi\colon|\operatorname{susp}\lk_\Delta(v)|\to|\Delta|$ carrying pieces
to pieces as follows. On $|c_+*\lk_\Delta(v)|$, let $\psi$ be the realization
of the simplicial isomorphism $c_+*\lk_\Delta(v)\to v*\lk_\Delta(v)$ which
sends $c_+$ to $v$ and fixes $\lk_\Delta(v)$. Then $\psi$ is a homeomorphism
onto $|v*\lk_\Delta(v)|$ on this piece. Since $|\Delta\setminus v|$ is
contractible, the inclusion
$|\lk_\Delta(v)|\hookrightarrow|\Delta\setminus v|$ is null-homotopic.
Hence it extends to a continuous map on the cone. On $|c_-*\lk_\Delta(v)|$,
let $\psi$ be such an extension. The two definitions agree on
$|\lk_\Delta(v)|$. Therefore, $\psi$ is well defined and continuous. The
restrictions of $\psi$ to the two cones are maps between contractible
spaces. Hence they are homotopy equivalences. Moreover, the restriction of
$\psi$ to the intersections is the identity map of $|\lk_\Delta(v)|$.
Therefore, by the gluing lemma, we get
$|\Delta|\simeq|\operatorname{susp}\lk_\Delta(v)|$. Since
$\lk_\Delta(v)\simeq S^0$ and unreduced suspension preserves homotopy
equivalences, we conclude that $\Delta\simeq\operatorname{susp}S^0=S^1$.
\end{proof}

\begin{lemma}\label{lem:interval}
Let $\Gamma$ be a graph on a finite linearly ordered vertex set
$u_1<\cdots<u_q$ with the property that, whenever $u_i$ and $u_j$ are adjacent
with $i<j$, the set $\{u_i,u_{i+1},\dots,u_j\}$ is a clique of $\Gamma$. If
$\Gamma$ is connected, then $\Cl(\Gamma)$ is contractible.
\end{lemma}

\begin{proof}
We proceed by induction on $q$. If $q=1$, then $\Cl(\Gamma)$ is a single
point, and the assertion is obvious.
Assume that $q\ge2$. We first show that consecutive vertices are adjacent.
Suppose that $u_iu_{i+1}\notin E(\Gamma)$ for some $i$. If
$u_au_b\in E(\Gamma)$ for some $a\le i<b$, then $\{u_a,\dots,u_b\}$ is a
clique of $\Gamma$. In particular, $u_i$ and $u_{i+1}$ are adjacent, which
is a contradiction. Hence no edge of $\Gamma$ joins $\{u_1,\dots,u_i\}$ to
$\{u_{i+1},\dots,u_q\}$. This contradicts the assumption that $\Gamma$ is
connected. Therefore, consecutive vertices are adjacent.

Set $f=\max\{j:u_1u_j\in E(\Gamma)\}$. Since consecutive vertices are
adjacent, we have $u_1u_2\in E(\Gamma)$. Hence $f$ is well defined and
$f\ge2$. Applying the hypothesis to the edge $u_1u_f$, we get that
$\{u_1,\dots,u_f\}$ is a clique. Hence $u_2,\dots,u_f$ are neighbors of
$u_1$. Moreover, by the maximality of $f$, the vertex $u_1$ has no other
neighbors. Therefore, the neighbors of $u_1$ are exactly $u_2,\dots,u_f$,
and they span a clique. This
implies that $\lk_{\Cl(\Gamma)}(u_1)$ is a nonempty simplex. Therefore, by
Lemma \ref{lem:deletion}(1), we get $\Cl(\Gamma)\simeq\Cl(\Gamma-u_1)$.
Note that the graph $\Gamma-u_1$, with the induced order
$u_2<\cdots<u_q$, satisfies the hypothesis of the lemma.
Moreover, $\Gamma-u_1$ is connected, since $u_2,\dots,u_q$ remain
consecutively adjacent. Now the assertion follows from the induction
hypothesis.
\end{proof}

\begin{lemma}\label{lem:arcrep}
Let $N\ge2$, suppose that every maximal clique of $H$ is an arc, and assume
\eqref{eq:star}. Then the following statements hold.
\begin{enumerate}[label=\textup{(\arabic*)}]
\item We have $N\ge2\omega-1$ and $\omega<N$. Every edge $uv$ belongs to a
clique of $H$ that is an arc, and in particular one of the two arcs with
endpoints $u,v$ is a clique.
\item For every $v\in\ZZ_N$ the closed neighborhood $N_H[v]$ is an arc, and
$|N_H[v]|\le N-\omega+1$.
\item Let $J$ be an arc with $|J|\le N-\omega+1$, written in its induced linear
order as $u_1<\cdots<u_p$. Then $H[J]$ is an indifference graph, that is,
there exist integers $f(1)\le\cdots\le f(p)$ such that, for $i<j$,
\[
u_iu_j\in E(H)\quad\Longleftrightarrow\quad j\le f(i).
\]
Consequently, $H[J]$ and all its induced subgraphs are chordal, and if
$J'\subseteq J$ is such that $H[J']$ is connected, then $\Cl(H[J'])$ is
contractible.
\end{enumerate}
\end{lemma}

\begin{proof}
(1) Choose a vertex $v$ in a clique of size $\omega$. Then
$|N_H[v]|\ge\omega$. Hence, by condition \eqref{eq:star} applied to $v$, we
get $2\omega\le|N_H[v]|+\omega\le N+1$. Therefore, $N\ge2\omega-1$. Since
$N\ge2$, this implies that $\omega\le(N+1)/2<N$. Note that every edge
belongs to a maximal clique, and every maximal clique is an arc. Such an arc
contains one of the two arcs joining the endpoints of the edge. Since a
subset of a clique is a clique, one of the two arcs with endpoints $u,v$ is
a clique.

(2) If $\omega=1$, then $H$ has no edges and $N_H[v]=\{v\}$, which is an arc
satisfying the bound. Assume that $\omega\ge2$. We claim that the closed
neighborhood $N_H[v]$ is the union of the maximal cliques containing $v$.
Indeed, $v$ lies in some maximal clique, every edge at $v$ lies in a maximal
clique containing $v$, and every clique containing $v$ lies in $N_H[v]$.
This proves the claim. These maximal cliques are arcs through the common
vertex $v$. Hence their union is either an arc or all of $\ZZ_N$. By
\eqref{eq:star}, we have $|N_H[v]|\le N-\omega+1$. Since $\omega\ge2$, this
gives $|N_H[v]|\le N-1$. This rules out the second possibility.

(3) Let $u,w\in J$ be adjacent. By part (1), there is an arc $A$ which is a
clique of $H$ and contains both $u$ and $w$. Since $A$ is an arc containing
$u$ and $w$, it contains one of the two arcs of $\ZZ_N$ that join $u$ to
$w$. Note that one of these two arcs lies inside $J$, since $J$ is an arc
containing $u$ and $w$. We call it the inner arc, and we call the other one
the outer arc. Since the two arcs together cover $\ZZ_N$ and the inner arc
lies inside $J$, the outer arc contains every vertex of
$\ZZ_N\setminus J$, together with $u$ and $w$ themselves. Hence it has at
least $N-|J|+2$ vertices. By hypothesis, $|J|\le N-\omega+1$. Therefore,
$N-|J|+2\ge\omega+1$. On the other hand, since $A$ is a clique, we have
$|A|\le\omega$. Hence $A$ does not contain the outer arc. Therefore, $A$
contains the inner arc. Since the inner arc is contained in the clique
$A$, the vertices of $J$ lying between $u$ and $w$ lie in $A$. Hence they
are pairwise adjacent and adjacent to both $u$ and $w$.

Define $f(i)=\max\bigl(\{j:u_iu_j\in E(H)\}\cup\{i\}\bigr)$. Then
$f(i)\ge i$ for every $i$. Moreover, if $f(i)>i$, then
$u_iu_{f(i)}\in E(H)$ by the definition of $f$. Applying the previous
paragraph to the pair $u_i,u_{f(i)}$, we get that $u_iu_j\in E(H)$ for
every $j$ with $i<j\le f(i)$. Conversely, if $u_iu_j\in E(H)$ with $i<j$,
then $j\le f(i)$ by the definition of $f$. This is the asserted
description of the edges. It remains to show that $f$ is nondecreasing.
Let $i<i'$. If $f(i)\le i'$, then $f(i)\le i'\le f(i')$ and there is
nothing to prove. Assume that $f(i)>i'$. Then $u_{i'}$ lies strictly
between $u_i$ and $u_{f(i)}$ in the linear order of $J$, and $u_i$ and
$u_{f(i)}$ are adjacent. Hence, by the previous paragraph, we get
$u_{i'}u_{f(i)}\in E(H)$. This implies that $f(i')\ge f(i)$. Therefore,
$H[J]$ is an indifference graph. It is known that indifference graphs are
chordal. Moreover, the defining property is inherited by induced
subgraphs. Hence the induced subgraphs of $H[J]$ are chordal as well.

Finally, by the first paragraph of this part, whenever two vertices of $J$
are adjacent, these two vertices together with the vertices of $J$ lying
between them form a clique. Hence $H[J]$ satisfies the hypothesis of Lemma
\ref{lem:interval} in its linear order. Since a subset of a clique is a
clique, this hypothesis passes to induced subgraphs on subsets of $J$.
Therefore, for every $J'\subseteq J$ such that $H[J']$ is connected, it
follows from Lemma \ref{lem:interval} that $\Cl(H[J'])$ is contractible.
\end{proof}

We are now ready to prove the main result of this section.

\begin{theorem}\label{thm:dichotomy}
Let $N\ge2$, suppose that every maximal clique of $H$ is an arc, and assume
\eqref{eq:star}. For every nonempty $W\subseteq\ZZ_N$, the complex $\Cl(H)[W]$
is homotopy equivalent to a disjoint union of $c(W)$ contractible complexes
when $c(W)\ge1$, and to $S^1$ when $c(W)=0$.
\end{theorem}

\begin{proof}
Suppose first that $c(W)\ge1$. Cutting at the uncovered links decomposes $W$
into maximal runs $W_1,\dots,W_{c(W)}$. If $c(W)=1$, then there is a single
run, equal to $W$ itself. We claim that vertices $x$ and $y$ belonging to
different runs, which requires $c(W)\ge2$, are nonadjacent. Note that each of
the two arcs of $\ZZ_N$ joining $x$ to $y$ contains a pair of consecutive
elements $w_t,w_{t+1}$ of $W$ whose link is uncovered, together with all of
$\ZZ_N$ between them. Hence each of these two arcs contains the forward arc
$[w_t,w_{t+1}]$ of an uncovered link. Suppose that $x$ and $y$ are adjacent.
By Lemma \ref{lem:arcrep}(1), there is a clique of $H$ which is an arc and
contains $x$ and $y$. We call such a clique a \emph{clique-arc}. This
clique-arc contains one of the two arcs joining $x$ to $y$. Hence it
contains the forward arc of an uncovered link. Since a subset of a clique
is a clique, the forward arc of an uncovered link is a clique, which is a
contradiction. This proves the claim.

Consider one run $W_r=\{u_0<\cdots<u_p\}$, in the linear order obtained
after cutting. Assume that $u_a\sim u_b$ with $a<b$. By Lemma
\ref{lem:arcrep}(1), there is a clique-arc containing $u_a$ and $u_b$, and
it contains one of the two arcs joining $u_a$ to $u_b$. As in the previous
paragraph, the arc that passes across a cut contains the forward arc of an
uncovered link. Hence it cannot be contained in a clique. Therefore, the
clique-arc contains the arc from $u_a$ to $u_b$ inside the run. In
particular, the vertices of $W_r$ between $u_a$ and $u_b$ form together
with $u_a$ and $u_b$ a clique. This shows that $H[W_r]$ satisfies the
hypothesis of Lemma \ref{lem:interval} in the cut order. Moreover, the
forward arc of a covered link is a clique containing its two endpoints.
Hence consecutive vertices of $W_r$ are adjacent, and $H[W_r]$ is
connected. It follows from Lemma \ref{lem:interval} that
$\Cl(H)[W_r]$ is contractible. By the claim, every clique of $H$ contained
in $W$ lies in a single run. Hence
\[
\Cl(H)[W]=\bigsqcup_{r=1}^{c(W)}\Cl(H)[W_r],
\]
and the first assertion follows.

Now assume that $c(W)=0$. We first show that $\omega\ge2$. Suppose that
$\omega=1$. Then $H$ has no edges. If $|W|\ge2$, then every forward arc
containing at least two vertices is uncovered. If $|W|=1$, then the forward
arc is $\ZZ_N$, which is not a clique, since $\omega<N$ by Lemma
\ref{lem:arcrep}(1). In both cases we get $c(W)\ge1$, which is a
contradiction. Therefore, $\omega\ge2$.

Since $c(W)=0$, every forward arc $[w_t,w_{t+1}]$ is a clique. Hence it has
at most $\omega$ vertices. This implies that the gaps satisfy
$g_t\le\omega-1$. Set $s=|W|$. Then $N=\sum_tg_t\le s(\omega-1)$.
Therefore,
\[
s\ge s_0=\left\lceil\frac{N}{\omega-1}\right\rceil .
\]
Moreover, since $N\ge2\omega-1$ by Lemma \ref{lem:arcrep}(1), we have
$N/(\omega-1)>2$. Hence $s_0\ge3$. In particular, $w_{t-1}$, $w_t$ and
$w_{t+1}$ are pairwise distinct for every $t$.

Next, we show that $W$ is not a clique. Suppose that $W$ is a clique. Then
$W$ lies in a maximal clique, which is a clique-arc $A$. Consider the
forward arc from the last vertex of $W$ in $A$ back to the first. Note that
these two vertices are consecutive in the cyclic order of $W$. Hence this
arc is the forward arc of a link of $W$. It contains $\ZZ_N\setminus A$
together with its two endpoints. Since $A$ is a clique, we have
$|A|\le\omega$. Hence the arc has at least $N-\omega+2$ vertices. By Lemma
\ref{lem:arcrep}(1), we have $N\ge2\omega-1$. Therefore, the arc has at
least $\omega+1$ vertices. This implies that it is not a clique, which
contradicts the assumption $c(W)=0$. Hence $W$ is not a clique.

Call an index $t$ \emph{bridgeable} if $[w_{t-1},w_{t+1}]$ is a clique. For
each $t$, write $v=w_t$ and set $N_W(v)=(W\cap N_H[v])\setminus\{v\}$. By
Lemma \ref{lem:arcrep}(2), the set $N_W(v)$ consists of consecutive vertices
of $W$ with $v$ removed, lying in the arc $N_H[v]$. We prove, by induction
on $s$, that $\Cl(H)[W]\simeq S^1$ for every $W$ with $c(W)=0$ and $|W|=s$.
By the bound above, every such $W$ satisfies $s\ge s_0$. The bridgeable case
below reduces $s$ by one and uses the induction hypothesis, and we show that
it cannot occur at $s=s_0$. The unbridgeable case yields $S^1$ directly,
without the induction hypothesis. In particular, it settles the base case
$s=s_0$.

Case 1: Suppose that some $t$ is bridgeable, and fix such a $t$. Since the
links $[w_{t-1},w_t]$ and $[w_t,w_{t+1}]$ are covered, their forward arcs
are cliques containing $v$. Hence $w_{t-1},w_{t+1}\in N_W(v)$, and
$N_W(v)$ is nonempty. We claim that consecutive elements of $N_W(v)$ are
adjacent. For pairs that remain consecutive in $W$, this is clear, since
the forward arc of a covered link is a clique containing its two
endpoints. The only new consecutive pair is $w_{t-1},w_{t+1}$, and this
pair is adjacent because $t$ is bridgeable. This proves the claim.
Therefore, $H[N_W(v)]$ is connected. Since $|N_H[v]|\le N-\omega+1$ by
Lemma \ref{lem:arcrep}(2), Lemma \ref{lem:arcrep}(3) applies to the arc
$N_H[v]$ and shows that $\lk(v)$ is contractible. Hence, by Lemma
\ref{lem:deletion}(1), we get
$\Cl(H)[W]\simeq\Cl(H)[W\setminus\{v\}]$. Note that deleting $v$ replaces
the two adjacent links by the covered link $[w_{t-1},w_{t+1}]$. Hence
$c(W\setminus\{v\})=0$, and the induction hypothesis applies. Moreover, this
case cannot occur when $s=s_0$. Indeed, if it did, then the deletion would
produce a set with $c=0$ and fewer than $N/(\omega-1)$ vertices. This
contradicts the inequality $N\le s(\omega-1)$ established above.

Case 2: Suppose that no index is bridgeable. Fix $t$ and again write
$v=w_t$. Split $N_W(v)=L\sqcup R$ into the vertices preceding and following
$v$ in the arc $N_H[v]$. Since the links $[w_{t-1},w_t]$ and
$[w_t,w_{t+1}]$ are covered, we have $w_{t-1}\in L$ and $w_{t+1}\in R$.
Hence both sets are nonempty. We claim that there is no edge between $L$
and $R$. Suppose that $x\in L$ and $y\in R$ are adjacent. By Lemma
\ref{lem:arcrep}(1), there is a clique-arc containing $x$ and $y$, and it
contains one of the two arcs joining $x$ to $y$. One of these two arcs
passes through $v$ and lies inside the arc $N_H[v]$. Since $w_{t-1}$ is the
last element of $W$ before $v$ and $w_{t+1}$ is the first element of $W$
after $v$, the arc through $v$ contains $[w_{t-1},w_{t+1}]$. Hence, if the
clique-arc contains the arc through $v$, then, since a subset of a clique
is a clique, $[w_{t-1},w_{t+1}]$ is a clique. This implies that $t$ is
bridgeable, which is a contradiction. Otherwise, the clique-arc contains
the other arc. The other arc contains $\ZZ_N\setminus N_H[v]$ together
with $x$ and $y$. By Lemma \ref{lem:arcrep}(2), we have
$|N_H[v]|\le N-\omega+1$. Hence the clique-arc has at least
$N-|N_H[v]|+2\ge\omega+1$ vertices. This is impossible, since a clique has
at most $\omega$ vertices, which is again a contradiction. This proves the
claim.

Note that consecutive elements of $L$ are consecutive in $W$. Since
$c(W)=0$, their links are covered. Hence they are adjacent, and $H[L]$ is
connected. The same argument applies to $R$. Therefore, by Lemma
\ref{lem:arcrep}(3) applied to the arc $N_H[v]$, we get
$\lk(v)=\Cl(H[L])\sqcup\Cl(H[R])\simeq S^0$.
Observe that deleting $v$ creates exactly one uncovered link, namely the
merged link $[w_{t-1},w_{t+1}]$. Hence, by the first part of the proof,
$\Cl(H)[W\setminus\{v\}]$ is contractible. Since $\lk(v)\simeq S^0$, we
conclude from Lemma \ref{lem:deletion}(2) that $\Cl(H)[W]\simeq S^1$.
\end{proof}

We now specialize to $H=C_n^m$, where $m$ and $n$ are integers with $m\ge1$
and $n\ge3m+1$. For the rest of the paper we
let $V=\ZZ_n$ with the cyclic distance
$d(x,y)=\min\{(x-y)\bmod n,\ (y-x)\bmod n\}$. For $x\in\ZZ_n$, the closed
$m$-ball centered at $x$ is
\[
\ball{x}=\{y\in\ZZ_n:d(x,y)\le m\},
\]
and it has $2m+1$ elements, since $n\ge2m+1$. An arc of $m+1$ consecutive
vertices, that is, a set of the form $\{i,i+1,\dots,i+m\}$ with $i\in\ZZ_n$,
is called a \emph{window}. Thus $x\sim y$ in $C_n^m$ if
and only if
$y\in\ball{x}\setminus\{x\}$, whereas $x\sim y$ in $\Gnm$ if and only if
$d(x,y)\ge m+1$. We write $\Dnm=\Ind(\Gnm)=\Cl(C_n^m)$.

\begin{lemma}\label{lem:windows}
Let $n\ge3m+1$. The facets of $\Dnm$ are the $n$ windows $\{i,\dots,i+m\}$,
$i\in\ZZ_n$, and $\dim\Dnm=m$.
\end{lemma}

\begin{proof}
By \cite[Lemma 3.6]{RPA}, every clique of $C_n^m$ is contained in an arc of
$m+1$ consecutive vertices. Note that each such window is itself a clique.
Indeed, two vertices of a window differ by at most $m$, and since
$n\ge2m+1$, their cyclic distance equals this difference. Moreover, since
all windows have the same number of elements and are pairwise distinct, no
window is contained in another. Hence the windows are precisely the facets,
and $\dim\Dnm=m$.
\end{proof}

\begin{corollary}\label{cor:dichotomy}
Let $n\ge3m+1$ and let $\emptyset\ne W\subseteq\ZZ_n$. If $c(W)\ge1$, then
$\Dnm[W]$ is homotopy equivalent to a disjoint union of $c(W)$ contractible
complexes. If $c(W)=0$, then $\Dnm[W]\simeq S^1$. Consequently,
\[
\dim_\kk\Hred_i(\Dnm[W];\kk)=
\begin{cases}
c(W)-1,& i=0,\ c(W)\ge1,\\
1,& i=1,\ c(W)=0,\\
0,&\text{otherwise.}
\end{cases}
\]
\end{corollary}

\begin{proof}
By Lemma \ref{lem:windows}, the maximal cliques of $C_n^m$ are arcs and
$\omega(C_n^m)=m+1$. Moreover, $|N_{C_n^m}[v]|=|\ball{v}|=2m+1$ for every
vertex $v$. Hence
$|N_{C_n^m}[v]|+\omega(C_n^m)=3m+2\le n+1$, so \eqref{eq:star} holds. Next,
we claim that a link is covered if and only if its gap is at most $m$. If
the gap is at most $m$, then the forward arc lies in a window. Hence it is
a clique. If the gap exceeds $m$, then the forward arc contains two
vertices at distance $m+1$, which are nonadjacent. Hence it is not a
clique. This proves the claim. Therefore, the assertion follows from
Theorem \ref{thm:dichotomy}.
\end{proof}

For $H=C_n^m$, the case $c(W)=0$ of Corollary \ref{cor:dichotomy}, together
with the vanishing of all positive-dimensional reduced homology in the case
$c(W)\ge1$, was proved by Rather, Pirzada and Aijaz
\cite[Proposition 3.8 and Lemma 3.5]{RPA} using the Nerve Lemma. Corollary
\ref{cor:dichotomy} determines, in addition, the number of connected
components of $\Dnm[W]$ when $c(W)\ge1$, and hence the dimension of
$\Hred_0(\Dnm[W];\kk)$.

\section{The linear strand, depth, and the Buchsbaum property}\label{sec:betti}

In this section, we compute the linear strand of the Betti table of
$R/I(\Gnm)$, and we prove that $R/I(\Gnm)$ is a Buchsbaum ring and that the
independence complex $\Dnm$ is a triangulated $m$-manifold with boundary
for $m\ge2$.

Let $C_m(n,k)$ denote the number of compositions of $n$ into $k$ ordered
parts belonging to $\{1,\dots,m\}$. By inclusion-exclusion over the parts
exceeding $m$,
\[
C_m(n,k)=\sum_{j\ge0}(-1)^j\binom{k}{j}\binom{n-jm-1}{k-1},
\]
where we use the convention that $\binom{a}{b}=0$ whenever $b<0$ or $a<b$.
Thus the sum is finite. The sets $W$ with $|W|=k$ and $c(W)=0$ were
enumerated in \cite[Lemma 3.3]{RPA}:
\begin{equation}\label{eq:zerocount}
\#\{W:|W|=k,\ c(W)=0\}=\frac{n}{k}\,C_m(n,k).
\end{equation}

\begin{lemma}\label{lem:counts}
Let $n\ge3m+1$ and $1\le k\le n$. Then
$\sum_{|W|=k}c(W)=n\binom{n-m-1}{k-1}$.
\end{lemma}

\begin{proof}
Since a link of $W$ is covered exactly when its gap is at most $m$, the sum
$\sum_{|W|=k}c(W)$ equals the number of pairs $(W,t)$ with $|W|=k$ and
$g_t>m$. We count these pairs. If $k=1$, then each singleton has a single
link, of gap $n>m$. Hence there are $n$ such pairs, which agrees with
$n\binom{n-m-1}{0}=n$. Assume that
$k\ge2$. Note that a pair $(W,t)$ is determined by the left endpoint $x$ of
the uncovered gap, its length $g\in\{m+1,\dots,n-k+1\}$, and a choice of the
remaining $k-2$ vertices among the $n-g-1$ vertices strictly between $x+g$
and $x$. Therefore, using the change of variable $h=n-g-1$ and the
hockey-stick identity, we obtain
\[
\sum_{|W|=k}c(W)=n\sum_{g=m+1}^{n-k+1}\binom{n-g-1}{k-2}
=n\sum_{h=k-2}^{n-m-2}\binom{h}{k-2}=n\binom{n-m-1}{k-1}. \qedhere
\]
\end{proof}

Rather, Pirzada and Aijaz obtained the second strand
$\beta_{k-2,k}(R/I)=\frac nkC_m(n,k)$ for $k\ge2$, the extremal Betti number
$\beta_{n-2,n}=1$, and the values $\reg(R/I)=2$ and $\pd(R/I)=n-2$
\cite[Theorems 3.10 and 3.11 and Corollary 3.12]{RPA}, and they determined
the linear strand recursively from the Hilbert series of the independence
complex \cite[Theorem 4.3]{RPA}. The following theorem gives the linear
strand in closed form.

\begin{theorem}\label{thm:linstrand}
Let $n\ge3m+1$ and $I=I(\Gnm)$. Then, for $1\le k\le n$,
\[
\beta_{k-1,\,k}(R/I)=n\binom{n-m-1}{k-1}-\binom{n}{k}+\frac{n}{k}\,C_m(n,k).
\]
Moreover, $\dpt R/I=2$ and $\dim R/I=m+1$.
\end{theorem}

\begin{proof}
By Theorem \ref{thm:hochster}, we have
$\beta_{k-1,k}(R/I)=\sum_{|W|=k}\dim_\kk\Hred_0(\Dnm[W];\kk)$. Hence, by
Corollary \ref{cor:dichotomy}, we get
\[
\beta_{k-1,k}(R/I)=\sum_{\substack{|W|=k\\ c(W)\ge1}}\bigl(c(W)-1\bigr)
=\sum_{|W|=k}c(W)-\binom{n}{k}+\#\{W:|W|=k,\ c(W)=0\},
\]
where the second equality holds because the sets with $c(W)=0$ contribute
nothing to the sum $\sum_{|W|=k}c(W)$. Therefore, the formula follows from
Lemma \ref{lem:counts} and
\eqref{eq:zerocount}. By \cite[Corollary 3.12]{RPA}, we have
$\pd(R/I)=n-2$. Therefore, the Auslander--Buchsbaum formula gives
$\dpt R/I=n-\pd(R/I)=2$. Finally, it follows from Lemma \ref{lem:windows}
that $\dim R/I=\dim\Dnm+1=m+1$.
\end{proof}

\begin{example}\label{ex:92}
For $n=9$ and $m=2$, Theorem \ref{thm:linstrand} and
\cite[Theorem 3.10]{RPA} give the Betti table
\[
\begin{array}{r|rrrrrrrr}
 & 0&1&2&3&4&5&6&7\\\hline
0& 1&\cdot&\cdot&\cdot&\cdot&\cdot&\cdot&\cdot\\
1& \cdot&18&51&54&18&\cdot&\cdot&\cdot\\
2& \cdot&\cdot&\cdot&9&30&27&9&1
\end{array}
\]
and therefore $\reg=2$, $\pd=7=n-2$, $\dpt=2$ and $\dim=3$. For $n=7,8$ and
$m=2$ the formulas reproduce the tables of \cite{RPS}, and for $n=12$, $m=3$
they reproduce the table computed from the Hilbert-series recursion in
\cite[Example 4.5]{RPA}. We have also checked Theorem \ref{thm:linstrand}
directly from Hochster's formula for $m=2$ with $n\le12$ and for $m=3$ with
$n\le13$.
\end{example}

We next study the local structure of $\Dnm$. The links of faces of $\Dnm$
are joins of simplices with the staircase complexes introduced in the
following lemma. The complex $\Sigma_3$ is drawn in Figure
\ref{fig:staircase}.

\begin{lemma}\label{lem:staircase}
For $r\ge1$, let $\Sigma_r$ be the complex on
$\ell_1,\dots,\ell_r,\rho_1,\dots,\rho_r$ with facets
$F_p=\{\ell_1,\dots,\ell_p\}\cup\{\rho_1,\dots,\rho_{r-p}\},
\qquad 0\le p\le r.$ Then $F_0,\dots,F_r$ is a shelling. Moreover,
$\Sigma_1=\{\ell_1\}\sqcup\{\rho_1\}\cong S^0$, while for $r\ge2$ the complex
$\Sigma_r$ is a shellable triangulated $(r-1)$-ball, and in particular is
contractible.
\end{lemma}

\begin{figure}[ht]
\centering
\begin{tikzpicture}[scale=1.05]
\coordinate (r3) at (0,0);
\coordinate (r2) at (1,1.1);
\coordinate (r1) at (2,0);
\coordinate (l1) at (3,1.1);
\coordinate (l2) at (4,0);
\coordinate (l3) at (5,1.1);
\fill[gray!25] (r3)--(r2)--(r1)--cycle;
\fill[gray!25] (r1)--(l1)--(l2)--cycle;
\draw[thick] (r3)--(r2)--(r1)--(l1)--(l2)--(l3);
\draw[thick] (r3)--(r1)--(l2);
\draw[thick] (r2)--(l1)--(l3);
\foreach \p in {r3,r2,r1,l1,l2,l3}{\fill (\p) circle (1.7pt);}
\node[below,font=\scriptsize] at (r3) {$\rho_3$};
\node[above,font=\scriptsize] at (r2) {$\rho_2$};
\node[below,font=\scriptsize] at (r1) {$\rho_1$};
\node[above,font=\scriptsize] at (l1) {$\ell_1$};
\node[below,font=\scriptsize] at (l2) {$\ell_2$};
\node[above,font=\scriptsize] at (l3) {$\ell_3$};
\node[font=\scriptsize] at (0.95,0.37) {$F_0$};
\node[font=\scriptsize] at (1.97,0.75) {$F_1$};
\node[font=\scriptsize] at (3.0,0.37) {$F_2$};
\node[font=\scriptsize] at (4.02,0.75) {$F_3$};
\end{tikzpicture}
\caption{The staircase complex $\Sigma_3$, a shellable triangulated
$2$-ball. Its four facets $F_0=\{\rho_1,\rho_2,\rho_3\}$,
$F_1=\{\ell_1,\rho_1,\rho_2\}$, $F_2=\{\ell_1,\ell_2,\rho_1\}$ and
$F_3=\{\ell_1,\ell_2,\ell_3\}$ are the consecutive triangles of the strip.}
\label{fig:staircase}
\end{figure}
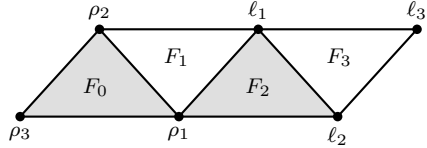

\begin{proof}
Note that the sets $F_0,\dots,F_r$ are pairwise incomparable. Indeed, for
$q<p$ we have $\ell_p\in F_p\setminus F_q$ and
$\rho_{r-q}\in F_q\setminus F_p$. Hence $F_0,\dots,F_r$ are precisely the
facets of $\Sigma_r$. Each facet $F_p$ has $r$ vertices. Therefore,
$\Sigma_r$ is pure of dimension $r-1$. Let $q<p$. Then
\[
F_p\cap F_q=\{\ell_1,\dots,\ell_q\}\cup\{\rho_1,\dots,\rho_{r-p}\}
\subseteq\{\ell_1,\dots,\ell_{p-1}\}\cup\{\rho_1,\dots,\rho_{r-p}\}
=F_p\cap F_{p-1},
\]
and $|F_p\cap F_{p-1}|=r-1$. Therefore,
$F_p\cap\bigcup_{q<p}F_q=F_p\cap F_{p-1}$. This shows that each new facet
meets the preceding complex in a single codimension-one face. Hence
$F_0,\dots,F_r$ is a shelling.

If $r=1$, then the two facets $F_0=\{\rho_1\}$ and $F_1=\{\ell_1\}$ are
disjoint. Hence $\Sigma_1\cong S^0$. Assume now that $r\ge2$. We show by
induction on $p$ that $\bigcup_{q\le p}F_q$ is contractible. For $p=0$,
this complex is the simplex $F_0$. For $p\ge1$, it is the union of the
contractible complex $\bigcup_{q<p}F_q$ and the simplex $F_p$ along
$F_p\cap F_{p-1}$, which is a simplex on $r-1\ge1$ vertices. Since the
union of two contractible complexes along a nonempty contractible
intersection is contractible, the induction proceeds. In particular,
$\Sigma_r$ is contractible.

It remains to verify the ball condition. Let $G\subset F_p$ with
$|G|=r-1$, and suppose that $G\subseteq F_q$. If $G=F_p\setminus\{\ell_p\}$
with $p\ge1$, then $\{\ell_1,\dots,\ell_{p-1}\}\subseteq F_q$ gives
$q\ge p-1$, and $\{\rho_1,\dots,\rho_{r-p}\}\subseteq F_q$ gives $q\le p$.
Hence $q\in\{p-1,p\}$. Similarly, if $G=F_p\setminus\{\rho_{r-p}\}$ with
$p\le r-1$, then $q\in\{p,p+1\}$. In every other case, $G$ is obtained by
deleting some $\ell_i$ with $i<p$ or some $\rho_i$ with $i<r-p$. In the
first case, $\ell_p\in G$ gives $q\ge p$, and $\rho_{r-p}\in G$ gives
$q\le p$ when $p<r$, while for $p=r$ the inequality $q\le r$ holds
trivially. In the second case, the same argument applies with the roles of
$\ell_p$ and $\rho_{r-p}$ exchanged. Hence $q=p$, and $G$ is contained only
in $F_p$. Therefore, every codimension-one face belongs to at most two
facets. Moreover, $F_0\setminus\{\rho_1\}=\{\rho_2,\dots,\rho_r\}$ is
nonempty, and it contains $\rho_r$, which lies only in $F_0$. Hence this
codimension-one face belongs to only one facet. Therefore, by Theorem
\ref{thm:dk}, we conclude that $\Sigma_r$ is a triangulated $(r-1)$-ball.
\end{proof}

\begin{theorem}\label{thm:buchsbaum}
Let $n\ge3m+1$. Then $R/I(\Gnm)$ is a Buchsbaum ring of dimension $m+1$ and
depth $2$, and it is Cohen--Macaulay if and only if $m=1$. Moreover,
$\Delta_{n,1}=C_n$ is a closed triangulated $1$-manifold, while for $m\ge2$ the
complex $\Dnm$ is a triangulated $m$-manifold with boundary, every vertex
lying on the boundary. When $m=2$ the complex is a M\"obius band if $n$ is
odd and an annulus if $n$ is even, its boundary being the $2$-regular graph
with edges $\{i,i+2\}$.
\end{theorem}

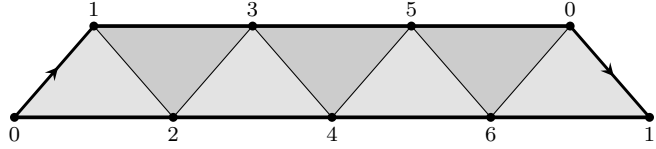
\begin{figure}[ht]
\centering
\begin{tikzpicture}[scale=1.05]
\foreach \i in {0,...,8}{
  \pgfmathtruncatemacro{\y}{mod(\i,2)}
  \coordinate (v\i) at (\i,{1.15*\y});
}
\foreach \i in {0,2,4,6}{
  \pgfmathtruncatemacro{\j}{\i+1}
  \pgfmathtruncatemacro{\k}{\i+2}
  \fill[gray!22] (v\i)--(v\j)--(v\k)--cycle;
}
\foreach \i in {1,3,5}{
  \pgfmathtruncatemacro{\j}{\i+1}
  \pgfmathtruncatemacro{\k}{\i+2}
  \fill[gray!40] (v\i)--(v\j)--(v\k)--cycle;
}
\foreach \i in {0,...,6}{
  \pgfmathtruncatemacro{\k}{\i+2}
  \draw[line width=1.4pt] (v\i)--(v\k);
}
\foreach \i in {0,...,7}{
  \pgfmathtruncatemacro{\j}{\i+1}
  \draw[thin] (v\i)--(v\j);
}
\draw[line width=1.1pt,
  decoration={markings, mark=at position 0.55 with {\arrow{stealth}}},
  postaction={decorate}] (v0)--(v1);
\draw[line width=1.1pt,
  decoration={markings, mark=at position 0.55 with {\arrow{stealth}}},
  postaction={decorate}] (v7)--(v8);
\foreach \i in {0,...,8}{\fill (v\i) circle (1.6pt);}
\node[below,font=\scriptsize] at (v0) {$0$};
\node[above,font=\scriptsize] at (v1) {$1$};
\node[below,font=\scriptsize] at (v2) {$2$};
\node[above,font=\scriptsize] at (v3) {$3$};
\node[below,font=\scriptsize] at (v4) {$4$};
\node[above,font=\scriptsize] at (v5) {$5$};
\node[below,font=\scriptsize] at (v6) {$6$};
\node[above,font=\scriptsize] at (v7) {$0$};
\node[below,font=\scriptsize] at (v8) {$1$};
\end{tikzpicture}
\caption{The complex $\Delta_{7,2}$ as a triangulated M\"obius band. The
seven triangles are the windows $\{i,i+1,i+2\}$ of $\ZZ_7$. Identifying the
two arrowed edges, both traversed from $0$ to $1$, closes the strip with a
half twist. The thick horizontal edges $\{i,i+2\}$, each lying in a single
triangle, form the boundary circle.}
\label{fig:mobius}
\end{figure}

\begin{proof}
Let $F$ be a nonempty face of $\Dnm$. By Lemma \ref{lem:windows}, $F$ is
contained in a window. Let $a$ and $b$ be the first and last elements of $F$
in that window, and set $r=m-(b-a)\ge0$. We claim that the windows
containing $F$ are exactly the sets $\{a-i,\dots,a-i+m\}$ with $0\le i\le
r$. Indeed, a window containing $a$ and $b$ contains one of the two arcs
joining them, and the arc from $b$ back to $a$ has $n-(b-a)+1\ge2m+2$
vertices, which exceeds $m+1$. Hence such a window contains the arc from
$a$ to $b$, and in particular contains $F$. Therefore, a window contains
$F$ if and only if it contains $a$ and $b$, that is, if and only if it is
of the form $\{a-i,\dots,a-i+m\}$ with $0\le i\le r$. This proves the
claim. Since every face of $\Dnm$ containing $F$ lies in a window
containing $F$, it lies in their union, the arc $[a-r,\,b+r]$. This arc has
$2m-(b-a)+1\le2m+1<n$ elements. Therefore, it is proper and carries a
linear order. Let $S$ be a set disjoint from $F$ with $S\cup F\in\Dnm$. By
the above, $S$ lies in the arc $[a-r,b+r]$. Moreover, for
$S\subseteq[a-r,b+r]\setminus F$, the set $S\cup F$ is a face of $\Dnm$ if
and only if it lies in a window containing $F$, and in the linear
coordinates of the arc this holds if and only if
$\max(S\cup F)-\min(S\cup F)\le m$. Decompose $S$ as
$S=S^0\sqcup S^-\sqcup S^+$, where $S^0\subseteq(a,b)\setminus F$,
$S^-\subseteq\{a-1,\dots,a-r\}$ and $S^+\subseteq\{b+1,\dots,b+r\}$. Set
$p=\max\{i:a-i\in S^-\}$ and $q=\max\{i:b+i\in S^+\}$, with
$\max\emptyset=0$. Then $\min(S\cup F)=a-p$ and $\max(S\cup F)=b+q$. Hence
the condition reads $(b-a)+p+q\le m$, that is, $p+q\le r$. In particular,
it imposes no restriction on $S^0$. Note that a subset of
$\{\ell_1,\dots,\ell_r,\rho_1,\dots,\rho_r\}$ lies in a facet $F_p$ of
$\Sigma_r$ if and only if its largest occurring indices $p'$ and $q'$
satisfy $p'+q'\le r$, since containment in $F_p$ means precisely that
$p'\le p$ and $q'\le r-p$. Therefore, identifying
$\ell_i=a-i$ and $\rho_i=b+i$, we obtain
\[
\lk_{\Dnm}(F)=\bigl(\text{simplex on }(a,b)\setminus F\bigr)*\Sigma_r,
\qquad \Sigma_0:=\{\emptyset\}.
\]
Suppose first that $r\ge2$. By Lemma \ref{lem:staircase}, $\Sigma_r$ is
contractible. Since a join with a contractible complex is contractible,
$\lk_{\Dnm}(F)$ is contractible as well. If $r=1$ and
$(a,b)\setminus F\ne\emptyset$, then the link has a nonempty simplex as a
join factor. Therefore, it is again contractible. If $r=1$ and
$(a,b)\setminus F=\emptyset$, then $\lk_{\Dnm}(F)=\Sigma_1\cong S^0$, which
has dimension $0$. In this case, Schenzel's condition requires only
$\Hred_{-1}=0$, which holds, since $S^0$ is nonempty. Finally, if $r=0$,
then the link is the simplex on $(a,b)\setminus F$. If this simplex is
nonempty, then the link is contractible. If it equals $\{\emptyset\}$,
which happens if and only if $F$ is a facet, then
$\dim\lk_{\Dnm}(F)=-1$ and the condition is vacuous. Hence the condition
is satisfied in this case as well. In every case, we get
$\Hred_i(\lk_{\Dnm}(F);\kk)=0$ for all $i<\dim\lk_{\Dnm}(F)$.

By Lemma \ref{lem:windows}, the complex $\Dnm$ is pure of dimension $m$.
Therefore, by Theorem \ref{thm:schenzel}, the ring $R/I$ is Buchsbaum. We
apply Corollary \ref{cor:dichotomy} to $W=\ZZ_n$. Every gap of $\ZZ_n$
equals $1\le m$. Hence $c(\ZZ_n)=0$, and we get $\Hred_0(\Dnm;\kk)=0$ and
$\Hred_1(\Dnm;\kk)\ne0$. Therefore, the depth formula of Theorem
\ref{thm:schenzel} gives $\dpt R/I=\min\{m+1,2\}=2$, in agreement with
Theorem \ref{thm:linstrand}. Since $\dim R/I=m+1$, the ring $R/I$ is
Cohen--Macaulay if and only if $m=1$.

We next consider the local topology. Taking $F=\{v\}$, we get $r=m$ and
$(a,b)\setminus F=\emptyset$. Hence $\lk_{\Dnm}(v)=\Sigma_m$. If $m=1$, then
this link is $S^0$. Therefore, $\Delta_{n,1}$ is a closed triangulated
$1$-manifold. In fact, $\Delta_{n,1}=\Cl(C_n)=C_n$. Assume that $m\ge2$. By
Lemma \ref{lem:staircase}, every vertex link is an $(m-1)$-ball. We use here
the link criterion of piecewise-linear topology. It states that a finite
simplicial complex in which the link of every vertex is a piecewise-linear
$(m-1)$-ball or $(m-1)$-sphere is a combinatorial $m$-manifold with
boundary, that its underlying space is a compact topological $m$-manifold
with boundary, and that the boundary is carried by the vertices whose link
is a ball (see \cite[Chapter 2]{RS} for combinatorial manifolds and this
criterion). The balls produced by Theorem \ref{thm:dk} are piecewise-linear, since
shellable balls are piecewise-linear balls \cite{DK}. Hence the links
$\Sigma_m$ are piecewise-linear $(m-1)$-balls, and they satisfy the
hypothesis of the criterion.
Therefore, $\Dnm$ is a combinatorial $m$-manifold with boundary, and every
vertex lies on the boundary.

It remains to identify the surface when $m=2$. The facets are the $n$
triangles $\{i,i+1,i+2\}$. Note that the edge $\{i,i+1\}$ belongs to the
two facets $\{i-1,i,i+1\}$ and $\{i,i+1,i+2\}$, whereas $\{i,i+2\}$ belongs
only to $\{i,i+1,i+2\}$. Since the boundary of a triangulated surface
consists of the edges lying in exactly one triangle, $\partial\Delta_{n,2}$
is the $2$-regular graph with edges $\{i,i+2\}$. This graph is a single
cycle when $n$ is odd and two cycles when $n$ is even, since the map
$i\mapsto i+2$ has one orbit on $\ZZ_n$ in the first case and two orbits in
the second. By Corollary \ref{cor:dichotomy}, we have
$\Delta_{n,2}\simeq S^1$. Hence the surface $\Delta_{n,2}$ is connected
and satisfies $\chi(\Delta_{n,2})=\chi(S^1)=0$. Compact connected surfaces
with boundary are classified by orientability, genus and the number $b$ of
boundary circles \cite[Chapter I]{Ma}, with $\chi=2-2g-b$ in the orientable
case and $\chi=2-g-b$, $g\ge1$, in the nonorientable case. If $\chi=0$ and
$b=1$, then the orientable case would give $2g=1$, which is impossible.
Hence the surface is nonorientable with $g=1$, that is, the M\"obius band.
If $\chi=0$ and $b=2$, then the nonorientable case would give $g=0<1$.
Hence the surface is orientable with $g=0$, that is, the annulus. Figure
\ref{fig:mobius} shows the case $n=7$.
\end{proof}

\section{Regularity of powers}\label{sec:powers}

In this section, we prove that the complement of every graph in the class of
Definition \ref{def:class} is locally linear, and we determine the
regularity of all powers of its edge ideal. We then specialize to $\Gnm$.

\begin{theorem}\label{thm:gen-loclin}
Let $N\ge2$, suppose that every maximal clique of $H$ is an arc, and assume
\eqref{eq:star}. Then $\overline H$ is locally linear.
\end{theorem}

\begin{proof}
Fix $v\in\ZZ_N$. If $\omega\ge2$, then Lemma \ref{lem:arcrep}(2) gives
$|N_H[v]|\le N-\omega+1\le N-1$. If $\omega=1$, then $H$ has no edges and
$N_H[v]=\{v\}$. In both cases, some vertex is nonadjacent to $v$ in $H$.
Hence $N_{\overline H}(v)\ne\emptyset$. Set
$K=\overline H-N_{\overline H}[v]$, that is, the induced subgraph of
$\overline H$ on the vertex set $N_H(v)$. We claim that
\[
(I(\overline H):x_v)=(x_u:u\in N_{\overline H}(v))+I(K).
\]
Indeed, for $u\in N_{\overline H}(v)$ we have $x_ux_v\in I(\overline H)$.
Hence $x_u$ lies in the colon ideal, and so does
$I(K)\subseteq I(\overline H)$. Conversely, let $w$ be a monomial with
$wx_v\in I(\overline H)$. Then $x_ax_b$ divides $wx_v$ for some edge
$\{a,b\}$ of $\overline H$. If $v\in\{a,b\}$, say $b=v$, then
$a\in N_{\overline H}(v)$ and $x_a$ divides $w$. If $v\notin\{a,b\}$, then
$x_ax_b$ divides $w$. In this case, either $\{a,b\}$ meets
$N_{\overline H}(v)$, and $w$ lies in the left summand, or
$a,b\in N_H(v)$, and $\{a,b\}$ is an edge of $K$. This proves the claim.
Note that the two summands involve disjoint sets of variables. If $K$ has
no edges, then the colon ideal is generated by the nonempty set of
variables on the left. Hence $\reg(I(\overline H):x_v)=1$. Assume that $K$
has an edge. Since the two summands lie in disjoint sets of variables,
$R/(I(\overline H):x_v)$ is the tensor product of a polynomial ring modulo
variables with a polynomial ring modulo $I(K)$. Since regularity is
additive on tensor products of $\kk$-algebras, this implies that
$\reg(I(\overline H):x_v)=\reg I(K)$. Note that the complement of $K$ on its
vertex set is $H[N_H(v)]$, which is an induced subgraph of $H[N_H[v]]$. By
parts (2) and (3) of Lemma \ref{lem:arcrep}, the graph $H[N_H[v]]$ is
chordal. Since chordality passes to induced subgraphs, $H[N_H(v)]$ is
chordal as well. Therefore, by Theorem \ref{thm:froberg}, we get
$\reg I(K)=2$. In every case, $\reg(I(\overline H):x_v)\le2$. Hence
$\overline H$ is locally linear.
\end{proof}

\begin{corollary}\label{cor:loclin}
Let $n\ge3m+1$. Then $\Gnm$ is locally linear.
\end{corollary}

\begin{proof}
By Lemma \ref{lem:windows}, the maximal cliques of $C_n^m$ are arcs, and
\eqref{eq:star} holds, as shown in the proof of Corollary
\ref{cor:dichotomy}. Hence the assertion follows from Theorem
\ref{thm:gen-loclin} applied to $H=C_n^m$.
\end{proof}

\begin{theorem}\label{thm:gen-powers}
Let $N\ge2$, suppose that every maximal clique of $H$ is an arc, and assume
\eqref{eq:star}. Put $I=I(\overline H)$. Then $\reg(R/I)\le2$,
$\indm(\overline H)\le2$, and
\[
\reg(I^k)=2k+\indm(\overline H)-1\qquad\text{for every }k\ge2.
\]
In particular, some $I^k$ with $k\ge2$ has a linear resolution if and only if
$\overline H$ is gap-free, and in that case every $I^k$ with $k\ge2$ has one.
\end{theorem}

\begin{proof}
By Theorem \ref{thm:dichotomy}, every induced subcomplex of
$\Ind(\overline H)=\Cl(H)$ on a nonempty vertex set is homotopy equivalent
to a disjoint union of contractible complexes or to $S^1$. Hence its
reduced homology is concentrated in degrees $0$ and $1$. The empty vertex
set contributes only $\beta_{0,0}(R/I)=1$. Therefore, by Theorem
\ref{thm:hochster}, apart from $\beta_{0,0}(R/I)=1$, we have
$\beta_{i,j}(R/I)\ne0$ only if $j-i-1\in\{0,1\}$, that is, only if
$j\in\{i+1,i+2\}$. This implies that
$\reg(R/I)\le2$. Hence, by Theorem \ref{thm:katzman}, we get
$\indm(\overline H)\le2$. Moreover, by Theorem \ref{thm:gen-loclin},
$\overline H$ is locally linear. Note that $H$ is not complete, since
$\omega<N$ by Lemma \ref{lem:arcrep}(1). Hence $\overline H$ has an edge.
This implies that $I\ne0$ and $\indm(\overline H)\ge1$. Since $I$ is
generated in degree two, we have $\beta_{1,2}(R/I)\ne0$. Therefore,
$\reg(R/I)\ge1$, and we conclude that $\reg(R/I)\in\{1,2\}$.

Case 1: Suppose that $\reg(R/I)=1$. Then $\reg I=2$. Hence, by Theorem
\ref{thm:froberg}, $H$ is chordal. Since $I$ is generated in degree two and
has a linear resolution, it follows from \cite{HHZ} that all powers of $I$
have linear resolutions. Since $I^k$ is generated in degree $2k$, this
gives $\reg(I^k)=2k$ for every $k\ge1$. Moreover, by Theorem
\ref{thm:katzman}, we get
$\indm(\overline H)\le\reg(R/I)=1$. Hence $\indm(\overline H)=1$ and
$\reg(I^k)=2k=2k+\indm(\overline H)-1$.

Case 2: Suppose that $\reg(R/I)=2$. Then $\reg I=3$. Since $\overline H$
is locally linear, the first part of Theorem \ref{thm:bbh} gives
$\reg(I^k)\le2k+\reg I-2=2k+1$. If $\indm(\overline H)=2$, then Theorem
\ref{thm:bht} gives the lower bound
$\reg(I^k)\ge2k+\indm(\overline H)-1=2k+1$. Hence $\reg(I^k)=2k+1$. If
$\indm(\overline H)=1$, then $\overline H$ is gap-free. Since $\overline H$
is also locally linear, the second part of Theorem \ref{thm:bbh} gives
$\reg(I^k)=2k$ for $k\ge2$. In both cases,
$\reg(I^k)=2k+\indm(\overline H)-1$.

Finally, the same formula shows that $\reg(I^k)=2k$ if and only if
$\indm(\overline H)=1$, that is, if and only if $\overline H$ is gap-free.
\end{proof}

For $\Gnm$ the induced matching number was computed in \cite{RPA}.

\begin{theorem}[{\cite[Proposition 3.1 and Corollary 3.13]{RPA}}]\label{thm:ind}
Let $n\ge3m+1$. Then
\[
\indm(\Gnm)=
\begin{cases}
2,&3m+1\le n\le4m,\\
1,&n\ge4m+1.
\end{cases}
\]
\end{theorem}

\begin{theorem}\label{thm:main}
Let $n\ge3m+1$ and let $I=I(\Gnm)$. Then $\reg I=3$ and
\[
\reg(I^k)=2k+\indm(\Gnm)-1\qquad\text{for all }k\ge2 .
\]
More explicitly,
\[
\reg(I^k)=
\begin{cases}
2k+1,&3m+1\le n\le4m,\\
2k,&n\ge4m+1.
\end{cases}
\]
In the first range the formula holds for every $k\ge1$, and in the second range
for every $k\ge2$. At $k=1$ in the second range $2k+\indm(\Gnm)-1=2$, whereas
$\reg I=3$.
\end{theorem}

\begin{proof}
By Lemma \ref{lem:windows}, the maximal cliques of $C_n^m$ are arcs, and
\eqref{eq:star} holds, as shown in the proof of Corollary
\ref{cor:dichotomy}. Hence Theorem \ref{thm:gen-powers} applies to
$H=C_n^m$. By \cite[Corollary 3.12]{RPA}, we have $\reg(R/I)=2$. Therefore,
$\reg I=3$. Combining Theorem \ref{thm:gen-powers} with Theorem
\ref{thm:ind}, we get $\reg(I^k)=2k+\indm(\Gnm)-1$ for every $k\ge2$. This
is the displayed case distinction. Finally, if $k=1$, then the value
$\reg I=3$ equals $2k+1$. Hence, in the range $3m+1\le n\le4m$, the formula
extends to every $k\ge1$.
\end{proof}

\begin{corollary}\label{cor:classification}
Let $n\ge3m+1$. Some power of $I(\Gnm)$ has a linear resolution if and only if
$n\ge4m+1$. In this case every $I^k$ with $k\ge2$ has a linear resolution.
\end{corollary}

\begin{proof}
Note that $I^k$ is generated in degree $2k$ for every $k\ge1$. Hence $I^k$
has a linear resolution if and only if $\reg(I^k)=2k$. If
$3m+1\le n\le4m$, then Theorem \ref{thm:main} gives $\reg(I^k)=2k+1$ for
every $k\ge1$. Therefore, no power of $I$ has a linear resolution in this
range. If $n\ge4m+1$, then Theorem \ref{thm:main} gives $\reg(I^k)=2k$ for
every $k\ge2$. Hence every such power has a linear resolution.
\end{proof}

For $m=2$, Theorems \ref{thm:ind} and \ref{thm:main} give
$\indm(\overline{C_n^2})=2$ for $n\in\{7,8\}$ and
$\indm(\overline{C_n^2})=1$ for $n\ge9$. The lower bounds in the first
value can be checked by
hand. Indeed, the edges $\{0,3\}$ and $\{1,5\}$ form an induced matching of
$G_{7,2}$, and the edges $\{0,4\}$ and $\{2,6\}$ form an induced matching of
$G_{8,2}$, while the reverse inequality $\indm\le2$ holds by Theorem
\ref{thm:gen-powers}. Consequently,
$\reg(I(\overline{C_7^2})^k)=\reg(I(\overline{C_8^2})^k)=2k+1$ for every
$k\ge1$, so no power of these two ideals has a linear resolution, while
$\reg(I(\overline{C_n^2})^k)=2k$ for $n\ge9$ and $k\ge2$. The article
\cite{RPS} treats $\overline{C_n^2}$ for $n>6$ as a family with induced
matching number one. The induced matchings above show that this holds only
for $n\ge9$, and the statements of \cite{RPS} that depend on gap-freeness,
in particular the regularity of powers in \cite[Theorem 44]{RPS}, are
therefore valid for $n\ge9$.

\begin{example}\label{ex:anticycle}
For $m=1$ and $n\ge5$ we have $G_{n,1}=\overline{C_n}$, and Theorem
\ref{thm:main} gives $\reg(I(\overline{C_n})^k)=2k$ for $k\ge2$. The
remaining case
$n=4$ lies in the first range of Theorem \ref{thm:main}, since $3m+1=4m=4$.
Here $G_{4,1}=\overline{C_4}=2K_2$, $\indm(G_{4,1})=2$ and
$\reg(I^k)=2k+1$.
\end{example}

For $m\ge2$ and $n\ge4m+1$, the linear resolutions of the
powers $I(\Gnm)^k$ are not covered by the known sufficient conditions of
\cite{Ne}, \cite{Ba} and \cite{Er}. The graph $\Gnm$ contains an induced
claw with center $m+3$ and leaves $0,1,2$, together with an induced cricket
with triangle $\{0,m+1,2m+2\}$ and pendant vertices $1,2$ attached at
$2m+2$. Hence the claw-free hypothesis, under which Nevo \cite{Ne} obtained
a linear resolution for the second power, and the cricket-free hypothesis,
under which Banerjee \cite{Ba} obtained linear resolutions for all $I^k$
with $k\ge2$, both fail for $\Gnm$. When $n\ge3m+4$, the graph also contains
an induced diamond on $\{0,1,m+2,2m+3\}$ whose missing edge is $\{0,1\}$.
Here the adjacency of $0$ and $2m+3$ amounts to
$d(0,2m+3)=n-2m-3\ge m+1$, that is, to the stated bound $n\ge3m+4$. Hence
the diamond-free criterion of Erey \cite{Er} fails as well. For $m\ge3$ the
inequality $4m+1\ge3m+4$ shows that this is automatic, and the single
exception in the range is $(n,m)=(9,2)$. A triangle of $G_{9,2}$ through $0$
requires two further vertices at cyclic distance at least $3$ from $0$ and
from each other, which forces the triangle $\{0,3,6\}$. Hence, by the
rotational symmetry of $G_{9,2}$, the triangles
of $G_{9,2}$ are the three pairwise disjoint sets $\{i,i+3,i+6\}$, no two
triangles share an edge, and $G_{9,2}$ is diamond-free. The linear
resolutions of its powers are therefore also covered by \cite{Er}.

The class of graphs in Definition \ref{def:class} is considerably larger
than the family of powers of cycles. The following lemma provides a general
construction of members of this class.

\begin{lemma}\label{lem:supply}
Let $s\ge4$. Let $A_0,\dots,A_{s-1}$ be arcs of $\ZZ_N$, indexed
cyclically, such that consecutive arcs intersect, arcs that are not
consecutive are disjoint, no arc contains another, and the arcs cover
$\ZZ_N$. Let $H$ be the graph on $\ZZ_N$ whose edges are the pairs of
distinct vertices lying in a common $A_i$. Then the maximal cliques of $H$
are exactly $A_0,\dots,A_{s-1}$, and $H$ satisfies \eqref{eq:star} if and
only if
\[
|A_i\cup A_{i+1}|+\max_j|A_j|\le N+1\qquad\text{for every }i.
\]
\end{lemma}

\begin{proof}
Note that each $A_i$ is a clique by construction. Since arcs that are not
consecutive are disjoint, two arcs with a common vertex are consecutive.
Moreover, no vertex lies in three arcs, since for $s\ge4$ the arcs
$A_{i-1}$ and $A_{i+1}$ are not consecutive and are therefore disjoint.
Hence every vertex lies in a single arc or in two consecutive arcs. Let $C$ be a clique of $H$ and let $u\in C$. Suppose first
that $u$ lies only in $A_i$. Then every vertex of $C$ shares an arc with
$u$. Hence $C\subseteq A_i$. Suppose next that $u\in A_i\cap A_{i+1}$. Then
$C\subseteq A_i\cup A_{i+1}$. We claim that $C$ cannot contain both a vertex
$x\in A_i\setminus A_{i+1}$ and a vertex $y\in A_{i+1}\setminus A_i$. Note
that the arcs containing $x$ are among $A_{i-1},A_i$, and the arcs
containing $y$ are among $A_{i+1},A_{i+2}$. Since $s\ge4$, the indices
$i-1,i,i+1,i+2$ are pairwise distinct modulo $s$. Hence $x$ and $y$ share no
arc. Therefore, $x$ and $y$ are nonadjacent, which proves the claim. It
follows that every clique lies in a single $A_i$. Since no arc contains
another, the maximal cliques are exactly the $A_i$.
Finally, we have $N_H[v]=A_i\cup A_{i+1}$ for $v\in A_i\cap A_{i+1}$, and
$N_H[v]=A_i$ for $v$ lying in the single arc $A_i$. Hence the displayed
condition is \eqref{eq:star} at the overlap vertices, which exist because
consecutive arcs intersect, and it implies \eqref{eq:star} at the remaining
vertices.
\end{proof}

Varying the number, the lengths and the pairwise overlaps of the arcs
produces a broad collection of irregular graphs in the class, and no
irregular graph is a power of a cycle, since $C_n^m$ is $2m$-regular. We
conclude this section with an explicit example.

\begin{example}\label{ex:z8}
Let $N=8$ and let $H$ be the graph of Lemma \ref{lem:supply} whose arcs,
and hence maximal cliques, are
\[
\{7,0,1\},\qquad \{1,2\},\qquad \{2,3,4\},\qquad \{4,5,6\},\qquad \{6,7\},
\]
shown in Figure \ref{fig:z8}. Here $\omega=3$ and the degree sequence is
$(2,3,3,2,4,2,3,3)$. Hence $H$ is not regular. Therefore, $H$ is not a power
of a cycle. Moreover, $\max_v|N_H[v]|+\omega=8\le N+1$. Hence \eqref{eq:star}
holds. Note that $H$ is not chordal, since the vertices $1,2,4,6,7$ induce a
five-cycle. Taking $W=\ZZ_8$, every link is covered. Hence, by Theorem
\ref{thm:dichotomy}, we get $\Cl(H)\simeq S^1$. Therefore, by Hochster's
formula, $\beta_{6,8}(R/I(\overline H))\ne0$. This implies that
$\reg(R/I(\overline H))\ge2$. Since Theorem \ref{thm:gen-powers} supplies
the matching upper bound, we conclude that $\reg(R/I(\overline H))=2$.
Checking the pairs of edges of $\overline H$, we get
$\indm(\overline H)=1$. Therefore, by Theorem \ref{thm:gen-powers}, we
obtain $\reg(I(\overline H)^k)=2k$ for every $k\ge2$.
\end{example}

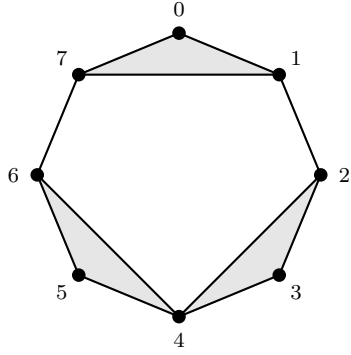
\begin{figure}[ht]
\centering
\begin{tikzpicture}[scale=1.5]
\foreach \i in {0,...,7}{
  \coordinate (h\i) at ({90-45*\i}:1.25);
}
\fill[gray!20] (h7)--(h0)--(h1)--cycle;
\fill[gray!20] (h2)--(h3)--(h4)--cycle;
\fill[gray!20] (h4)--(h5)--(h6)--cycle;
\foreach \i in {0,...,7}{
  \pgfmathtruncatemacro{\j}{mod(\i+1,8)}
  \draw[thick] (h\i) -- (h\j);
}
\draw[thick] (h7)--(h1);
\draw[thick] (h2)--(h4);
\draw[thick] (h4)--(h6);
\foreach \i in {0,...,7}{
  \fill (h\i) circle (1.7pt);
  \node[font=\scriptsize] at ({90-45*\i}:1.46) {$\i$};
}
\end{tikzpicture}
\caption{A circular interval graph on $\ZZ_8$ that is not a power of a
cycle. Its maximal cliques are the two shaded triangles $\{2,3,4\}$ and
$\{4,5,6\}$, the triangle $\{7,0,1\}$, and the edges $\{1,2\}$ and
$\{6,7\}$.}
\label{fig:z8}
\end{figure}

\section{Even-connection graphs and colon ideals}\label{sec:colon}

In this section, we prove Theorem \ref{thm:key}. Throughout this section we
assume that $n\ge4m+1$, and we write $G=\Gnm$ and $I=I(G)$. Let
$M=e_1\cdots e_k$ be a minimal monomial generator of $I^k$, with
$e_j=\{a_j,b_j\}\in E(G)$, so that $d(a_j,b_j)\ge m+1$ for every $j$. Set
$A_j=\ball{a_j}$ and $B_j=\ball{b_j}$, and define
\[
\cB=\{\ball{c}:c\in\{a_1,b_1,\dots,a_k,b_k\}\},
\qquad
\cM(u)=\{X\in\cB:u\notin X\}\ \ (u\in\ZZ_n).
\]
The \emph{span} of an arc is one less than its number of vertices. We begin
with two lemmas concerning arcs in $\ZZ_n$.

\begin{lemma}\label{lem:lindist}
Let $J\subseteq\ZZ_n$ lie in an arc of span at most $n-m-1$. If $u,v\in J$
are at linear distance $\delta$ in the linear order of that arc, then
$d(u,v)\le m$ if and only if $\delta\le m$, and in
that case $d(u,v)=\delta$.
\end{lemma}

\begin{proof}
Since $J$ lies in an arc of span at most $n-m-1$, we have $\delta\le n-m-1$.
Therefore, $n-\delta\ge m+1$. If $\delta\le m$, then
$n-\delta\ge n-m>m\ge\delta$. Hence $d(u,v)=\delta\le m$. If $\delta>m$,
then both $\delta$ and $n-\delta$ are greater than $m$. Hence $d(u,v)>m$.
\end{proof}

\begin{lemma}\label{lem:arcs}
Let $n\ge4m+1$, and let $X=\ball{c}$ and $Y=\ball{c'}$ be closed $m$-balls.
Then the following statements hold.
\begin{enumerate}[label=\textup{(\arabic*)}]
\item The intersection $X\cap Y$ is an arc, possibly empty. If it is nonempty,
then $|X\cap Y|=2m+1-d(c,c')$ and $d(c,c')\le2m$.
\item If $u,v\in X$ and $d(u,v)\le m$, then the shorter arc joining $u$ and $v$
is contained in $X$.
\item Suppose in addition that $d(c,c')\ge m+1$, that $X\cap Y\ne\emptyset$ and
that $X\cup Y\ne\ZZ_n$. Then $d(w,z)\ge m+2$ for every $w\in X\cap Y$ and every
$z\notin X\cup Y$.
\end{enumerate}
\end{lemma}

\begin{proof}
(1) Suppose that $X\cap Y$ has two components. We claim that this can occur
only when $X\cup Y=\ZZ_n$ and neither ball contains the other. Indeed, if
$X\cup Y\ne\ZZ_n$, then both balls lie in the complement of a point, hence
in a common arc, and the intersection of two subarcs of an arc is an arc.
Moreover, if one ball contains the other, then the intersection is that
ball, which is an arc. This proves the claim. In that case,
$n=|X\cup Y|=4m+2-|X\cap Y|.$
Since $n\ge4m+1$, we obtain $|X\cap Y|\le1$. This is incompatible with two
nonempty components. Hence every nonempty intersection is a single arc.

Assume that $X\cap Y\ne\emptyset$. We claim that the two balls can overlap
only on one side. Indeed, a two-sided overlap gives $d(c,c')\le2m$ and
$n-d(c,c')\le2m$. This implies that $n\le4m$, which contradicts
$n\ge4m+1$. This proves the claim. Now, in suitable coordinates with $c=0$
and $c'=d(c,c')$, the intersection is the arc $[d(c,c')-m,\,m]$. Hence
$|X\cap Y|=2m+1-d(c,c')$. In particular, since the intersection is
nonempty, $d(c,c')\le2m$.

(2) Note that the ball $X$ is an arc with $2m+1$ vertices. Since
$u,v\in X$, one of the two arcs joining $u$ and $v$ lies in $X$. The other
arc has at least $n-m+1\ge3m+2$ vertices, which is more than $|X|=2m+1$.
Therefore, the arc contained in $X$ is the shorter one.

(3) Without loss of generality, after reversing the cyclic orientation if
necessary, we may assume that $c=0$ and $c'=t$ with $t=d(c,c')$, and
$t\le2m$ by part (1). Then
$X\cap Y=[t-m,m]$ and $X\cup Y=[-m,t+m]$. Since $X\cup Y\ne\ZZ_n$, the
complement $\ZZ_n\setminus(X\cup Y)=[t+m+1,n-m-1]$ is nonempty. Let
$w\in[t-m,m]$ and $z\in[t+m+1,n-m-1]$. Observe that the forward distance
from $w$ to $z$ is at least $(t+m+1)-m=t+1$, and the forward distance from
$z$ to $w$ is at least $(t-m)-(n-m-1)+n=t+1$. Therefore, we get
$d(w,z)\ge t+1\ge m+2$, since $t\ge m+1$ by hypothesis. This completes the
proof.
\end{proof}

\begin{definition}\label{def:transfer}
A pair $(c,c')\in\ZZ_n\times\ZZ_n$ is called a \emph{transfer pair} for $M$ if
there exist $\ell\ge1$ and vertices $p_1,\dots,p_{2\ell}$ with $p_1=c$ and
$p_{2\ell}=c'$ such that the following conditions hold.
\begin{enumerate}[label=\textup{(\roman*)}]
\item Each $\{p_{2i-1},p_{2i}\}$ is one of the edges $e_j$.
\item No edge $e_j$ is used more often than its multiplicity in $M$.
\item $d(p_{2i},p_{2i+1})\ge m+1$ for every $1\le i\le\ell-1$.
\end{enumerate}
We denote the set of all transfer pairs by $\cP(M)$, and we set
$\cQ(M)=\{(\ball{c},\ball{c'}):(c,c')\in\cP(M)\}$.
\end{definition}

In a transfer pair $(c,c')$, the vertex $c=p_1$ lies in the edge
$\{p_1,p_2\}$, which is one of the $e_j$, and likewise $c'=p_{2\ell}$.
Hence $c$ and $c'$ are among the letters $a_1,b_1,\dots,a_k,b_k$, and both
$\ball{c}$ and $\ball{c'}$ belong to $\cB$. Note that reversing a defining
sequence shows that $\cP(M)$, and hence $\cQ(M)$, is symmetric. Moreover,
taking $\ell=1$, we get
$(A_j,B_j),(B_j,A_j)\in\cQ(M)$ for every $j$.

\begin{lemma}\label{lem:transfer}
Let $u$ and $v$ be vertices of $\ZZ_n$, not necessarily distinct. Then $u$
and $v$ are even-connected with respect to $M$ if and only if
\[
\cQ(M)\cap\bigl(\cM(u)\times\cM(v)\bigr)\ne\emptyset .
\]
\end{lemma}

\begin{proof}
Suppose first that $u$ and $v$ are even-connected, and let
$u=p_0,p_1,\dots,p_{2\ell+1}=v$ be a sequence as in Definition
\ref{def:evenconn}. We show that the interior sequence $p_1,\dots,p_{2\ell}$
satisfies the conditions of Definition \ref{def:transfer}. Note that each
$\{p_{2i-1},p_{2i}\}$ with $1\le i\le\ell$ is one of the $e_j$, and no $e_j$
is used beyond its multiplicity in $M$. Moreover, for $1\le i\le\ell-1$, the
pair $\{p_{2i},p_{2i+1}\}$ is an edge of $G$. Hence
$d(p_{2i},p_{2i+1})\ge m+1$. Therefore, $(p_1,p_{2\ell})\in\cP(M)$ and
$(\ball{p_1},\ball{p_{2\ell}})\in\cQ(M)$. In particular, both $\ball{p_1}$
and $\ball{p_{2\ell}}$ belong to $\cB$. Moreover, $\{p_0,p_1\}$ and
$\{p_{2\ell},p_{2\ell+1}\}$ are edges of $G$. Hence $d(u,p_1)\ge m+1$ and
$d(v,p_{2\ell})\ge m+1$. In other words, $u\notin\ball{p_1}$ and
$v\notin\ball{p_{2\ell}}$. Therefore, $\ball{p_1}\in\cM(u)$ and
$\ball{p_{2\ell}}\in\cM(v)$, and the displayed intersection is nonempty.

Conversely, let $(X,Y)\in\cQ(M)\cap\bigl(\cM(u)\times\cM(v)\bigr)$, and
write $X=\ball{c}$ and $Y=\ball{c'}$ for a transfer pair $(c,c')\in\cP(M)$
with defining sequence $p_1=c,p_2,\dots,p_{2\ell}=c'$. The conditions
$X\in\cM(u)$ and $Y\in\cM(v)$ mean that $u\notin\ball{c}$ and
$v\notin\ball{c'}$. In other words, $d(u,p_1)\ge m+1$ and
$d(p_{2\ell},v)\ge m+1$. Consider the sequence
$u=p_0,p_1,\dots,p_{2\ell},p_{2\ell+1}=v$. Note that its pairs
$\{p_{2i-1},p_{2i}\}$ with $1\le i\le\ell$ are edges $e_j$ of $G$ used
within their multiplicities in $M$. Moreover, its remaining consecutive
pairs, namely $\{p_0,p_1\}$, the pairs $\{p_{2i},p_{2i+1}\}$ with
$1\le i\le\ell-1$, and $\{p_{2\ell},p_{2\ell+1}\}$, join vertices at
distance at least $m+1$. Hence they are edges of $G$. Therefore, the
sequence satisfies the conditions of Definition \ref{def:evenconn}, and $u$
and $v$ are even-connected with respect to $M$.
\end{proof}

Let $G'$ and $S$ be the graph on $\ZZ_n$ and the set defined after Theorem
\ref{thm:banerjeecolon}, for the present $G$ and $M$. By Theorem
\ref{thm:banerjeecolon}, the
ideal $(I^{k+1}:M)$ is generated in degree two. Hence a monomial of degree
two lies in it if and only if it is one of its minimal generators.
Therefore, Theorem \ref{thm:banerjeecolon} and Lemma
\ref{lem:transfer} give
\begin{equation}\label{eq:Gprime}
E(G')=E(G)\cup\bigl\{\{u,v\}:u\ne v,\
\cQ(M)\cap(\cM(u)\times\cM(v))\ne\emptyset\bigr\},
\end{equation}
and
\begin{equation}\label{eq:S}
S=\bigl\{u:\cQ(M)\cap(\cM(u)\times\cM(u))\ne\emptyset\bigr\}.
\end{equation}
In particular, since $E(G)\subseteq E(G')$, every edge of $\overline{G'}$
joins two vertices at distance at
most $m$.

\begin{lemma}\label{lem:Siso}
Let $n\ge4m+1$ and let $u\in S$. Then $u$ is isolated in $\overline{G'}$.
\end{lemma}

\begin{proof}
Since $u\in S$, it follows from \eqref{eq:S} and Lemma \ref{lem:transfer}
that there is an even-connection $u=p_0,p_1,\dots,p_{2\ell+1}=u$
of $u$ with itself. Let $v\ne u$ satisfy $d(u,v)\le m$.

Suppose first that $d(p_{2i},v)\ge m+1$ for some $1\le i\le\ell$. Then
$u=p_0,\dots,p_{2i},v$ is an even-connection from $u$ to $v$. Indeed,
$\{p_{2i},v\}$ is an edge of $G$, and the sequence uses only a
sub-multiset of the edges appearing in $M$. Next, suppose that
$d(p_{2i+1},v)\ge m+1$ for some $0\le i\le\ell$. Note that the reversed
sequence $u=p_{2\ell+1},p_{2\ell},\dots,p_0=u$ is again an even-connection
of $u$ with itself, since the defining conditions are preserved under
reversal. Moreover, $i\le\ell-1$, since
$d(p_{2\ell+1},v)=d(u,v)\le m$. In the reversed sequence, the vertex
$p_{2i+1}$ therefore occupies the even position $2(\ell-i)\ge2$. Hence, as
in the first case, appending $v$ gives an
even-connection between $u$ and
$v$.

It remains to consider the case $d(p_i,v)\le m$ for every $i$. In this case,
the whole walk lies inside the ball $\ball{v}$. Note that the span of
$\ball{v}$ is $2m\le n-m-1$. Hence, by Lemma \ref{lem:lindist}, for points
inside $\ball{v}$, the condition of having cyclic distance at least $m+1$
is detected by the linear order of this arc. Choose coordinates
$\ball{v}=[-m,m]$ with $v=0$. Observe that two points both in $[0,m]$, or
both in $[-m,0]$, have linear distance at most $m$. Hence they have cyclic
distance at most $m$. This implies that two points of $[-m,m]$ at cyclic
distance at least $m+1$ lie strictly on opposite sides of $0$.

Now, observe that every consecutive pair $p_i,p_{i+1}$ in the
even-connection is an edge of $G$. Hence $d(p_i,p_{i+1})\ge m+1$, and the
signs of the $p_i$ alternate. On the other hand, $p_0=p_{2\ell+1}=u$, while the indices $0$ and
$2\ell+1$ have opposite parity. This implies that the two occurrences of
$u$ have opposite signs, which is a contradiction. Therefore, this case
cannot occur.

Hence $u$ is even-connected to every $v\ne u$ with $d(u,v)\le m$. By
\eqref{eq:Gprime}, this implies that no such vertex is adjacent to $u$ in
$\overline{G'}$. Since
every edge of $\overline{G'}$ joins vertices at distance at most $m$, we
conclude that $u$ is isolated in $\overline{G'}$.
\end{proof}

\begin{lemma}\label{lem:local}
Let $n\ge4m+1$. Every edge of $\overline{G'}$ has both endpoints in $A_1$ or
both endpoints in $B_1$.
\end{lemma}

\begin{proof}
Let $\{u,v\}$ be an edge of $\overline{G'}$. Then $d(u,v)\le m$, and, by
\eqref{eq:Gprime}, $u$ and
$v$ are not even-connected with respect to $M$. Recall that $A_1=\ball{a_1}$
and $B_1=\ball{b_1}$, where $\{a_1,b_1\}$ is the first edge of $M$, and that
$(A_1,B_1)$ and $(B_1,A_1)$ both lie in $\cQ(M)$. Suppose that
$u\notin A_1$ and $v\notin B_1$. Then the pair $(A_1,B_1)$ lies in
$\cM(u)\times\cM(v)$. Hence, by Lemma \ref{lem:transfer}, $u$ and $v$ are
even-connected, which is a contradiction. Therefore, $u\in A_1$ or
$v\in B_1$. Applying the same argument to $(B_1,A_1)$, we also get
$v\in A_1$ or $u\in B_1$.

Expanding the two disjunctions gives four cases. If $u\in A_1$ and
$v\in A_1$, then $\{u,v\}\subseteq A_1$. If $v\in B_1$ and $u\in B_1$, then
$\{u,v\}\subseteq B_1$. In the two mixed cases, either $u\in A_1$ and
$u\in B_1$, which gives $u\in A_1\cap B_1$, or $v\in B_1$ and $v\in A_1$,
which gives $v\in A_1\cap B_1$. Hence either the edge already lies in one
of the two balls, in which case we are done, or one of its endpoints lies
in $A_1\cap B_1$. After exchanging the names of $u$ and $v$ if necessary,
we may assume that $u\in A_1\cap B_1$.

If $A_1\cup B_1=\ZZ_n$, then $v\in A_1\cup B_1$ trivially. Assume that
$A_1\cup B_1\ne\ZZ_n$. Since $u\in A_1\cap B_1$, the intersection
$A_1\cap B_1$ is nonempty. Moreover, $d(a_1,b_1)\ge m+1$. Hence Lemma
\ref{lem:arcs}(3) gives $d(u,w)\ge m+2$ for every $w\notin A_1\cup B_1$. On
the other hand, $d(u,v)\le m$. Hence $v\in A_1\cup B_1$ in this case as
well. Since $u$ belongs to both balls, we conclude that the edge $\{u,v\}$
lies entirely in $A_1$ or entirely in $B_1$.
\end{proof}

\begin{lemma}\label{lem:elim}
Let $n\ge4m+1$, and let $J\subsetneq\ZZ_n$ be an arc such that both
$A_1\cap J$ and $B_1\cap J$ are intervals, possibly empty, in the induced
linear order of $J$.
Then the largest element of $J$ is simplicial in $\overline{G'}[J]$.
Consequently, listing the elements of $J$ in decreasing order gives a perfect
elimination ordering, and in particular $\overline{G'}[J]$ is chordal.
\end{lemma}

\begin{proof}
Let $\{u,w\}$ be an edge of $\overline{G'}[J]$. By Lemma \ref{lem:local},
its two endpoints lie either in $A_1$ or in $B_1$. Hence they lie in the
interval $A_1\cap J$ or in the interval $B_1\cap J$. Each of these intervals
has at most $2m+1$ elements. Therefore, the linear distance $\delta$ between
$u$ and $w$ in $J$ is at most $2m$. Since $J$ is a proper arc, the two arcs
joining $u$ and $w$ have $\delta$ and $n-\delta$ steps, and
$n-\delta\ge n-2m\ge2m+1>\delta$. Hence the cyclic distance of $u$ and $w$
equals $\delta$. Moreover, since $E(G)\subseteq E(G')$ by
\eqref{eq:Gprime}, the vertices $u$ and $w$ are not adjacent in $G$. Hence
this distance is at most $m$.

Let $v=\max J$. If $v$ has at most one neighbor in $\overline{G'}[J]$, then
$v$ is simplicial. Assume otherwise, and let $u_1<u_2<v$ be two neighbors of
$v$ in
$\overline{G'}[J]$. By the first paragraph, we have $v-u_1=d(u_1,v)\le m$.
Since $u_2$ lies between $u_1$ and $v$ in the linear order of $J$, it lies
on the segment of $J$ from $u_1$ to $v$. This segment has $d(u_1,v)$
steps. Hence it is the shorter arc joining $u_1$ to $v$, and $u_2$ lies on
that shorter arc. By Lemma \ref{lem:arcs}(2), every
ball of $\cB$ containing both $u_1$ and $v$ contains that shorter arc.
Hence it contains $u_2$. In other words, a ball of $\cB$ that omits $u_2$
omits $u_1$ or omits $v$. This implies that
$\cM(u_2)\subseteq\cM(u_1)\cup\cM(v).$
Moreover, $u_2-u_1<v-u_1\le m$. Hence, as in the first paragraph,
$d(u_1,u_2)=u_2-u_1\le m$.

Suppose that $u_1$ and $u_2$ are even-connected. Then, by Lemma
\ref{lem:transfer}, there is a pair $(X,Y)\in\cQ(M)$ with $X\in\cM(u_1)$
and $Y\in\cM(u_2)$. By the inclusion above, either $Y\in\cM(v)$ or
$Y\in\cM(u_1)$. If $Y\in\cM(v)$, then
$(X,Y)\in\cQ(M)\cap\bigl(\cM(u_1)\times\cM(v)\bigr)$. Hence, by Lemma
\ref{lem:transfer}, $u_1$ and $v$ are even-connected. Therefore, by
\eqref{eq:Gprime}, $u_1$ and $v$ are not adjacent in $\overline{G'}$, which
is a contradiction. If $Y\in\cM(u_1)$, then
$(X,Y)\in\cQ(M)\cap\bigl(\cM(u_1)\times\cM(u_1)\bigr)$. Hence, by
\eqref{eq:S}, we get $u_1\in S$. Therefore, by Lemma \ref{lem:Siso}, the
vertex $u_1$ is isolated in $\overline{G'}$. This contradicts the fact that
$u_1$ is adjacent to $v$.

Therefore, $u_1$ and $u_2$ are not even-connected. Since $d(u_1,u_2)\le m$,
it follows from \eqref{eq:Gprime} that $\{u_1,u_2\}$ is an edge of
$\overline{G'}$, and hence of $\overline{G'}[J]$, since $u_1,u_2\in J$.
Consequently, the
neighbors of $v$ form a clique, and $v$ is simplicial. Note that, when
$|J|\ge2$, the
hypotheses of the lemma are preserved after removing the largest element of
$J$. Indeed, $J\setminus\{v\}$ is again a proper arc, and removing the
largest element of $J$ removes at most the largest element of each of the
two intervals. Therefore, repeating the argument gives a perfect
elimination ordering in decreasing order. Since a graph with a perfect
elimination ordering is chordal, $\overline{G'}[J]$ is chordal.
\end{proof}

\begin{theorem}\label{thm:key}
Let $n\ge4m+1$, let $I=I(\Gnm)$ and let $k\ge1$. If $M$ is a minimal monomial
generator of $I^k$, then $\reg(I^{k+1}:M)=2$.
\end{theorem}

\begin{proof}
Let $G''$ be the graph obtained from $G'$ by attaching a pendant vertex
$u'$ to every $u\in S$. Then $(I^{k+1}:M)^{\mathrm{pol}}=I(G'')$, as noted
after Theorem \ref{thm:banerjeecolon}. Since
polarization preserves regularity \cite[Corollary 1.6.3]{HHbook}, it
suffices, by Theorem
\ref{thm:froberg}, to show that $\overline{G''}$ is chordal.

We first prove that $\overline{G'}$ is chordal. By Lemma \ref{lem:local},
every edge of $\overline{G'}$ lies inside $A_1$ or inside $B_1$.

Case 1: Suppose that $A_1\cap B_1=\emptyset$. Since every edge of
$\overline{G'}$ lies inside $A_1$ or inside $B_1$, the graph
$\overline{G'}$ is the
disjoint union of $\overline{G'}[A_1]$, $\overline{G'}[B_1]$ and isolated
vertices. Note that both $A_1$ and $B_1$ are proper arcs whose
intersections with the other ball are empty. Hence, by Lemma
\ref{lem:elim}, both induced subgraphs are chordal. Since a disjoint union
of chordal graphs is chordal, $\overline{G'}$ is chordal.

Case 2: Suppose that $A_1\cap B_1\ne\emptyset$ and that $L=A_1\cup B_1$ is
not all of $\ZZ_n$. Then $L$ is an arc. Indeed, both balls lie in the
complement of a point, and the union of two intersecting subarcs of an arc
is an arc. Moreover, every edge of $\overline{G'}$ lies
in $L$, and $A_1$ and $B_1$, being arcs contained in the proper arc $L$,
are intervals in the linear order of $L$.
Therefore, by Lemma \ref{lem:elim} applied to $L$, the graph
$\overline{G'}[L]$ is chordal. Since all remaining vertices are isolated,
$\overline{G'}$ is chordal.

Case 3: Suppose that $A_1\cap B_1\ne\emptyset$ and that
$A_1\cup B_1=\ZZ_n$. Then
$n=|A_1\cup B_1|=4m+2-|A_1\cap B_1|$. Since $n\ge4m+1$ and
$|A_1\cap B_1|\ge1$, this implies that
$n=4m+1$ and $|A_1\cap B_1|=1$, say $A_1\cap B_1=\{c\}$. By Lemma
\ref{lem:elim} applied separately to $A_1$ and to $B_1$, the graphs
$\overline{G'}[A_1]$ and $\overline{G'}[B_1]$ are chordal. Since every edge
of $\overline{G'}$ lies inside $A_1$ or inside $B_1$, no edge joins
$A_1\setminus\{c\}$ to $B_1\setminus\{c\}$. We claim that $\overline{G'}$ is
chordal. Indeed, a cycle meeting both $A_1\setminus\{c\}$ and
$B_1\setminus\{c\}$ would have to pass through $c$ at least twice, which is
impossible. Hence every cycle of $\overline{G'}$ lies in $A_1$ or in $B_1$.
In particular, an induced cycle of length at least four would be an induced
cycle of $\overline{G'}[A_1]$ or of $\overline{G'}[B_1]$, which is
impossible. This proves the claim.

We now show that $\overline{G''}$ is chordal. Suppose, to the contrary, that
$\overline{G''}$ contains an induced cycle $C$ of length $r\ge4$. If every
vertex of $C$ lies in $\ZZ_n$, then $C$ is an induced cycle of
$\overline{G'}$, since $G''$ and $G'$ have the same edges inside $\ZZ_n$.
This contradicts the chordality proved above. Hence $C$
contains a pendant vertex $u'$ corresponding to some $u\in S$.

Since $u$ is the only neighbor of $u'$ in $G''$, the vertex $u'$ is
adjacent in $\overline{G''}$ to every vertex except $u$. Hence $u'$ has only
one non-neighbor. On the other hand, on an induced cycle of length $r$
through $u'$, the vertex $u'$ is nonadjacent to exactly $r-3$ other
vertices of the cycle. Hence $r-3\le1$, and therefore $r=4$. The vertex
opposite to $u'$ on the cycle is nonadjacent to $u'$. Hence it equals
$u$. Therefore, we may write the cycle as $u',x,u,y$, where $x$ and
$y$ are adjacent to $u$ in $\overline{G''}$ and $x\not\sim y$.

By Lemma \ref{lem:Siso}, the vertex $u$ is isolated in $\overline{G'}$.
Hence the neighbors of $u$ in $\overline{G''}$ are pendant vertices.
Moreover, $u'$ is not a neighbor of $u$. Hence these neighbors correspond
to elements of $S\setminus\{u\}$. Write $x=w'$ and $y=z'$.
Note that the vertex $w'$ is adjacent in $\overline{G''}$ to every vertex
except $w$, and $z'\ne w$, since $z'$ does not lie in $\ZZ_n$. Therefore,
$w'$ and $z'$ are adjacent. This contradicts $x\not\sim y$.

Hence $\overline{G''}$ is chordal. Note that $MI\subseteq I^{k+1}$. Hence
$I\subseteq(I^{k+1}:M)$, and $I(G'')$ is nonzero. Therefore, by Theorem
\ref{thm:froberg}, we get $\reg I(G'')=2$. Since polarization preserves
regularity, as noted above, we conclude that
$\reg(I^{k+1}:M)=2$.
\end{proof}

Theorem \ref{thm:key} yields a direct proof of the equality
$\reg(I^k)=2k$ for $k\ge2$ in Theorem \ref{thm:main}. Indeed, by a result
of Banerjee \cite[Theorem 5.2]{Ba}, we have
\[
\reg(I^{k+1})\le\max\bigl\{\reg(I^k),\ \max_M\reg(I^{k+1}:M)+2k\bigr\}
\qquad\text{for every }k\ge1,
\]
where $M$ runs over the minimal monomial generators of $I^k$. By Theorem
\ref{thm:key}, the inner maximum equals $2k+2$. Since $\reg I=3$ by
\cite[Corollary 3.12]{RPA}, induction on $k$ gives $\reg(I^k)\le2k$ for
every $k\ge2$. On the other hand, $I^k$ is generated in degree $2k$. Hence
$\reg(I^k)=2k$ for every $k\ge2$.

\end{document}